\documentclass[11pt,psamsfonts]{article}
\usepackage{a4,fullpage,amssymb,epsf,psfrag,times}
\usepackage{graphicx}
\usepackage{amsmath}
\usepackage{amsfonts}
\usepackage{enumerate}
\usepackage{latexsym}
\usepackage{epsfig}
\usepackage{amsthm}
\usepackage{amsmath}
\usepackage{amssymb}
\usepackage{latexsym}
\usepackage{epsfig}
\usepackage{amssymb,latexsym}
\usepackage[all]{xy}

\makeatletter\@addtoreset{equation}{section}\makeatother
\makeatletter\@addtoreset{figure}{section}\makeatother
\makeatletter\@addtoreset{table}{section}\makeatother

\newtheorem{theorem}{Theorem}[section]
\newtheorem{prop}[theorem]{Proposition}
\newtheorem{lemma}[theorem]{Lemma}
\newtheorem{claim}[theorem]{Claim}

\newcommand{\R}{{\mathbb R}}

\newcommand{\Z}{{\mathbb Z}}

\newcommand{\N}{{\mathbb N}}
\newcommand{\T}{{\mathbb T}}

\newcommand{\phy}{\varphi}
\newcommand{\trsp}{\raisebox{.6ex}{${\scriptstyle t}$}}

\newcommand{\op}[1]{\!\!\mathop{\rm ~#1}\nolimits}
\newcommand{\DD}{\!\mathop{\rm d\!}\nolimits}

\newcommand{\scriptop}[1]{\!\!\mathop{\mbox{\rm \scriptsize ~#1}}\nolimits}

\newenvironment{remark}{\par\medskip\noindent{\bf
Remark~\thetheorem~~}}{\unskip\nobreak\hfill\hbox{ $\oslash$}\par\bigskip}

\newenvironment{definition}{\refstepcounter{theorem}\par\medskip\noindent{\bf
Definition~\thetheorem~~}}{\unskip\nobreak\hfill\hbox{ $\oslash$}\par\bigskip}

\newcommand{\got}[1]{\mathfrak{#1}}

\title{
Semitoric integrable systems on symplectic $4$\--manifolds}
\author{Alvaro Pelayo\thanks{Partially supported by an NSF Postdoctoral Fellowship} \,\,
and San V\~u Ng\d oc}
\date{}
\begin{document}
\maketitle

\begin{abstract}
   Let $(M,\, \omega)$ be a symplectic $4$\--manifold.
 A semitoric integrable system on $(M,\, \omega)$
 is a pair of smooth functions $J,\,H \in \op{C}^{\infty}(M,\,\R)$ 
  for which
   $J$ generates a Hamiltonian $S^1$\--action
   and the Poisson brackets $\{J,\,H\}$ vanish. 
We shall introduce new global 
  symplectic invariants for these systems; some of
  these invariants
  encode topological or geometric aspects, while others encode
  analytical information about the singularities and
  how they stand with respect to the system.
  Our goal is to prove that a semitoric system is
  completely determined by the invariants we introduce. 
   \end{abstract}


\section{Introduction}

Atiyah \cite[Th.~1]{atiyah} and Guillemin\--Sternberg \cite{gs}
proved that the image $\mu(M)$ under the momentum map 
$
\mu:=(\mu_1,\ldots,\mu_n) \colon M \to \R^n
$
of a Hamiltonian action of an $n$\--dimensional torus
on a compact connected symplectic manifold $(M, \, \omega)$
is a convex polytope, called the \emph{momentum polytope}.  Delzant
\cite{delzant} showed that if the dimension of the torus is half
the dimension of $M$, the momentum polytope, which in this
case is called {\em Delzant polytope}, determines the
isomorphism type of $M$. 
Moreover, he showed that $M$ is 
a toric variety.
These  theorems establish remarkable and deep connections
bewteen Hamiltonian dynamics, symplectic geometry, K\"ahler manifolds
and toric varieties in algebraic geometry. Through the analysis of the
quantization of such systems, one may also mention important
links with the representation theory of Lie groups and Lie algebras,
semiclassics, and microlocal analysis.

Nevertheless, at least from the viewpoint of symplectic geometry, the
situation described by the momentum polytope is very
rigid. There are at least three natural directions for further
mathematical exploration~: (i) replacing the manifold $M$ by an orbifold;
(ii) allowing more general actions than Hamiltonian ones,
(iii) replacing the torus $T$ by a non\--abelian and/or non-compact Lie group $G$.

Following (i) Lerman\--Tolman generalized Delzant's classification to
orbifolds in \cite[Th.~7.4, 8.1]{lermantolman}.
Regarding (i) Pelayo generalized Delzant's result to
the case when $M$ is $4$\--dimensional and $T$ acts symplectically, but
not necessarily Hamiltonianly \cite[Th.~8.2.1]{pelayo}. 
This result relies on the
generalization of Delzant's theorem for symplectic torus actions with
coisotropic principal orbits by Duistermaat\--Pelayo
earlier \cite[Th.~9.4,\,9.6]{DuPe}, and for
symplectic torus actions with symplectic principal orbits 
by Pelayo \cite[Th.~7.4.1]{pelayo}.
Regarding (iii), results for non\--abelian compact Lie groups $G$ are relatively
complete, see Kirwan \cite{kirwan}, Lerman\--Meinrenken\--Tolman\--Woodward \cite{lermanmtw},
Sjamaar \cite{sjamaar} and Guillemin\--Sjamaar \cite{guis}.
When $T$ is replaced by a non\--compact group $G$ the theory
is hard; even in the proper and Hamiltonian case,
the symplectic local normal form
for a proper action requires extensive work, see Marle \cite{marle}
and Guillemin\--Sternberg \cite[Sec. 41]{gsst}; in the non\--Hamiltonian
symplectic case this normal form is recent work of
Benoist \cite[Prop. 1.9]{benoist} and 
Ortega\--Ratiu \cite{ortegaratiu}.

The seemingly most simple non\--compact case to study is that of a
Hamiltonian action of the abelian group $\R^n$ on a $2n$\--dimensional
symplectic manifold. But of course, this is nothing less than the goal of the
theory of integrable systems. The role of the momentum map is in this case
played by a map of the form
$
F:=(f_1,\,\dots,\,f_n) \colon M \to \R^n, 
$
where $f_i \colon M\to\R$ is smooth, the Poisson
brackets $\{f_i,\,f_j\}$ identically vanish on $M$, and the
differentials $\DD{}f_1,\,\dots,\,\DD{}f_n$ are almost\--everywhere
linearly independent.
In this article
we study the case of an
integrable system $f_1:=J,\, f_2:=H$, where $M$ is $4$\--dimensional
and the 
component $J$ generates a Hamiltonian $S^1$\--action: these
are called semitoric.
Semitoric systems form an important class of integrable
systems, commonly found in simple physical models. Indeed, a semitoric
system can be viewed as a Hamiltonian system in the presence of an $S^1$
symmetry~\cite{sadovski-zhilinski}. 
One of the incentives for this work
is that it is much simpler to understand the integrable system on its
whole rather than writing a theory of Hamiltonian systems on
Hamiltonian $S^1$\--manifolds. 

It is well established in the integrable systems community that the
most simple and natural object, which tells much about the structure
of the integrable system under study, is the so\--called
bifurcation diagram. This is nothing but the image in $\R^2$ of $F=(J,\,H)$ or,
more precisely, the set of critical values of $F$. In this article, we
are going to show that the arrangement of such critical values is
indeed important, but other crucial ingredients are needed to understand
$F$, which are
more subtle and cannot be detected from the bifurcation diagram itself.
Our goal is
to construct a collection of
new global symplectic invariants for semitoric integrable
systems which completely determine a semitoric system up 
to isomorphisms. 
We will build on a number of  remarkable results by other authors in
 integrable systems,
  including Arnold, Atiyah, Dufour\--Molino, Eliasson, Duistermaat, Guillemin\--Sternberg, Miranda\--Zung and  V\~u Ng\d oc, to which we shall make references throughout the text, and to whom
  this paper owes much credit.

The paper is structured as follows; 
in Section~\ref{sec:semitoric} we define semitoric systems, explain the conditions
which appear in the definition and announce our main result;
in sections~\ref{sec:analytic},\ref{sec:polytope} and~\ref{sec:geometric} we construct the
new symplectic invariants. Specifically, in Section~\ref{sec:analytic} we study the  analytical invariants, in
Section~\ref{sec:polytope} we study the combinatorial invariants, and in Section~\ref{sec:geometric} we study the geometric
invariants. In Section~\ref{sec:main} we state the aforementioned
theorem, which we prove in Section~\ref{sec:proof}. The paper concludes with a short 
appendix, Section \ref{sec:miranda}, in which we prove
a very slight modification of a result of Miranda\--Zung which
we need earlier.

\section{Semitoric systems}
\label{sec:semitoric}
First we introduce the precise definition of semitoric integrable system.

\begin{definition} \label{semitoricdef} Let $(M, \, \omega)$ be a
  connected symplectic $4$\--dimensional manifold.  A {\em
    semitoric integrable system} on $M$ is an integrable system
  $J, \, H \in \op{C}^{\infty}(M,\, \R)$ for which
  \begin{itemize}
  \item[(1)] the component $J$ is a proper
    momentum map for a Hamiltonian circle action on $M$;
  \item[(2)] the map $F:=(J,\,H):M\to\R^2$ has only non-degenerate
    singularities in the sense of Williamson, without real-hyperbolic
    blocks.
  \end{itemize}
  We also use the terminology {\em $4$\--dimensional semitoric integrable system} to
  refer to the triple $(M,\, \omega,\, (J,\,H))$.
\end{definition}

We recall that the first point in Definition \ref{semitoricdef}
means that the preimage by $J$ of a
compact set is compact in $M$ (which is of course automatic if $M$ is
compact), and the second point means that, whenever $m$ is a critical
point of $F$, there exists a 2 by 2 matrix $B$ such that, if we denote
$\tilde{F}=B\circ F$, one of the following happens, in some local
symplectic coordinates near $m$~:
\begin{itemize}
\item[(1)] $\tilde{F}(x,\,y,\,\xi,\,\eta)=(\eta + \mathcal{O}(\eta^2),\,x^2+\xi^2
  + \mathcal{O}((x,\,\xi)^3))$
\item[(2)] $\op{d}^2_m\tilde{F}(x,\,y,\,\xi,\,\eta)=(x^2+\xi^2,\,y^2+\eta^2)$
\item[(3)] $\op{d}^2_m\tilde{F}(x,\,y,\,\xi,\,\eta)=(x\xi+y\eta,\,x\eta-y\xi)$
\end{itemize}
The first case is called a transversally --- or codimension 1 ---
{\em elliptic singularity}; the second case is an {\em elliptic\--elliptic
singularity}; the last case is a {\em focus\--focus singularity}.

In \cite{vungoc}, V\~u Ng\d{o}c proved a version of the 
Atiyah\--Guillemin\--Sternberg
theorem:
to a $4$\--dimensional semitoric integrable system
one may meaningfully associate a family of convex polygons which
generalizes the momentum polygon that one has in the presence of a
Hamiltonian $2$\--torus action. 
If two such systems are isomorphic, then these two families of polygons
are equal. 

In view of this, a natural goal is to try to understand
whether a 
semitoric integrable system on a symplectic $4$\--manifold could possibly be determined
by this family of polygons; as it turns out this is one of five invariants we associate to such a system.
Precisely, the invariants are the following: 
(i) {\em the number of singularities
    invariant}: an integer counting the number of isolated singularities;
(ii)
  {\em the singularity type invariant}: which classifies
  locally the type of singularity; 
 (iii)
  {\em the polygon invariant}: a family of weighted rational convex
  polygons (generalizing
  the Delzant polygon and which may be viewed as a bifurcation diagram);
(iv)
{\em the volume invariant}: numbers
  measuring volumes of certain submanifolds at the singularities;
(v)
{\em the twisting index invariant}: 
  integers measuring how twisted the system is around singularities.
Our goal in this paper is to prove an integrable system
is completely determined, up to isomorphisms, by these invariants.  In
other words, we shall prove that:
\begin{center}
  \emph{$(M,\, \omega_1,\,(J_1,\,H_1))$ and $(M,\, \omega_2,\, (J_2,\,H_2))$ are isomorphic $\iff$ they
    have the same invariants (i)--(v).}
\end{center}
Here the word \emph{isomorphism} is used in the sense that there
exists a symplectomorphism
$$
\varphi \colon M_1 \to M_2,\,\,\,\,\,\,\, \textup{such that}\,\,\,\,\,
\varphi^*(J_2,\,H_2)=(J_1,\,f(J_1,\,H_1)).
$$
for some smooth function $f$ (see Theorem~\ref{mainthm}).

One could say that (i) and (ii) are analytical invariants, (iii) is a
combinatorial/group\--theoretic invariant, and (iv), (v) are geometric
invariants.  

\section{Analytic invariants of a semitoric system}
\label{sec:analytic}

We describe invariants of a semitoric system encoding analytic
information about the singularities. 
Throughout this section
$(M,\, \omega,\, (J,\,H))$ is a $4$\--dimensional
semitoric integrable system.

\subsection{Cardinality of singular set invariant}
\label{mfsection}

It is clear from the definition that a semitoric integrable system has
only two types of singularities: elliptic (of codimension $0$ or $1$) and
focus\--focus. This can easily be inferred from the bifurcation diagram. 
In fact, V\~u Ng\d oc proves in~\cite[Prop.~2.9, Th.~3.4, Cor.~5.10]{vungoc}
the following statement~: 

\begin{prop} \label{integer:prop}
The semitoric system $(M,\, \omega,\, (J,\,H))$ admits a finite
number $m_f$ of focus\--focus critical values $c_1,\dots,c_{m_f}$, and,
denoting by $B=F(M)$ the image of $F$, where $F=(J,\,H)$:
\begin{itemize}
\item[(a)]
the set of regular values of $F$ is 
$
B_{r}=\op{Int}{B}\setminus\{c_1,\dots,c_{m_f}\};
$
\item[(b)]
the boundary of $B$ consists of all images of elliptic
singularities;
\item[(c)]
the fibers of $F$ are connected.
\end{itemize}
\end{prop}

Of course $m_f$ is an invariant of the singular foliation induced by
$F$, where $m_f$ and $F$ are as in Proposition \ref{integer:prop}.
Since this foliation is preserved by isomorphism, we have the
following result.

\begin{lemma} \label{i)}
Let $(M_1, \, \omega_1, (J_1,\,H_1)),\, (M_2, \, \omega_2, (J_2,\,H_2))$ be
isomorphic
$4$\--dimensional semitoric integrable systems and let $m_f^i$ be the number
of focus\--focus points of $(M_i, \, \omega_i, (J_i,\,H_i))$, where $i \in \{1,\,2\}$. Then $m^1_f=m^2_f$.
\end{lemma}

One may argue that $m_f$ is a combinatorial invariant, since it is an integer;
we have put it in this section because we
need it for the construction of the true analytic invariant of the system,
defined in Section \ref{focussection}: the singularity type invariant.

\subsection{Singularity type invariant}
\label{focussection}

Let $F$, $m_f$ and $c_1,\ldots,c_{m_f}$ be as in Proposition \ref{integer:prop}
We consider here the preimage by $F$ of a focus\--focus critical value
$c_i$, where $i \in \{1,\,\ldots,\,m_f\}$.
Throughout, we will assume that the critical fiber
$$
\mathcal{F}_m:=F^{-1}(c_i)
$$
 contains only one critical point $m$. According to
Zung~\cite{ntz1}, this is a genericity assumption.

\begin{figure}[htb]
\begin{center}
\includegraphics{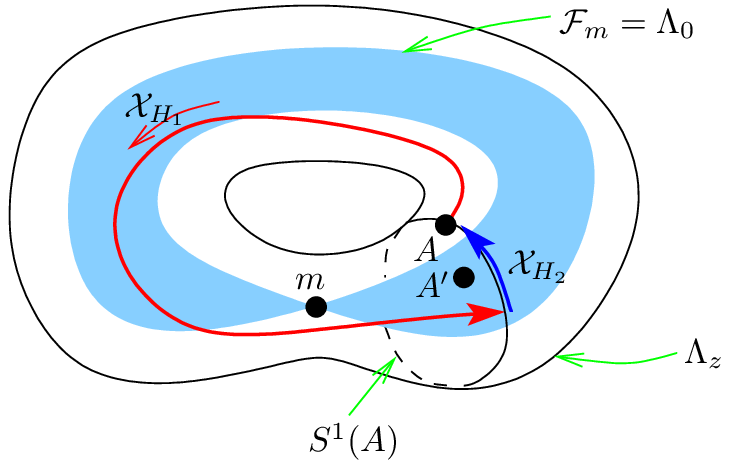}
\caption{The singular foliation $\mathcal{F}$ associated to $F$ near the singular leaf $\mathcal{F}_m$,
where $S^1(A)$ denotes the $S^1$\--orbit for the $S^1$\--action generated by $H_2$.}
\end{center}
\label{AF4}
\end{figure}

Let $\mathcal{F}$ denote the associated singular foliation.
Eliasson's theorem \cite{eliasson-these} describes a neighborhood $U$
of a focus\--focus point $m$ in a singular foliation of focus\--focus
type: there exist symplectic coordinates $(x,\, y,\, \xi,\,\eta)$ in
$U$ in which the map $(q_1,\,q_2)$, given by
\begin{equation}
  q_1=x\xi+y\eta,\qquad q_2=x\eta-y\xi, \label{equ:cartan}
\end{equation}
is a momentum map for the foliation $\mathcal{F}$; here
the critical point $m$ corresponds to coordinates $(0,\,0,\,0,\,0)$.
Let us fix a point
$A'\in \mathcal{F}_m\cap (U\setminus\{m\})$, let $\Sigma$ denote a small
2\--dimensional surface transversal to $\mathcal{F}$ at the point $A'$,
and let $\Omega$ be the open neighborhood of the leaf $\mathcal{F}_m$
which consists of the leaves which intersect the surface $\Sigma$.

Since the Liouville foliation in a small neighborhood of $\Sigma$ is
regular for both momentum maps $F$ and $q=(q_1,\,q_2)$, there must be
a local diffeomorphism $\varphi$ of $\R^2$ such that $q=\varphi \circ
{F}$, and hence we can define a global momentum map $\Phi=\varphi
\circ{F}$ for the foliation, which agrees with $q$ on $U$.
Write $\Phi:=(H_1,\,H_2)$ and $\Lambda_z:=\Phi^{-1}(z)$.
Note that
$
\Lambda_0=\mathcal{F}_m.
$ 
It follows from~(\ref{equ:cartan}) that near $m$
the $H_2$\--orbits must be periodic of primitive period $2\pi$ for any
point in a (non-trivial) trajectory of $\mathcal{X}_{H_1}$.

Suppose that $A \in\Lambda_z$ for some
regular value $z$.
We define $\tau_1(z)$, which is a strictly positive number, as the
time it takes the Hamiltonian flow associated to $H_1$ leaving from
$A$ to meet the Hamiltonian flow associated to $H_2$ which passes
through $A$, and let $\tau_2(z)\in\R/2\pi\Z$ the time that it takes
to go from this intersection point back to $A$, hence closing the
trajectory. Write $z=(z_1,\,z_2)=z_1+\op{i}z_2$, and let $\op{ln} z$
for a fixed determination of the logarithmic function on the complex
plane. We moreover define the following two functions:
\begin{equation}
  \left\{
    \begin{array}{ccl}
      \sigma_1(z) & = & \tau_1(z)+\Re(\op{ln} z) \\
      \sigma_2(z) & = & \tau_2(z)-\Im(\op{ln} z),
    \end{array}
  \right.
  \label{equ:sigma}
\end{equation}
where $\Re$ and $\Im$ respectively stand for the real
an imaginary parts of a complex number.
In his article \cite[Prop.\,3.1]{vungoc0}, V\~u Ng\d oc proved that
$\sigma_1$ and $\sigma_2$ extend to smooth and single\--valued
functions in a neighbourhood of $0$ and that the differential 1\--form
\begin{eqnarray} \label{sigma}
\sigma:=\sigma_1\, \DD{}z_1+\sigma_2\, \DD{}z_2
\end{eqnarray}
is closed.
Notice that if follows from the smoothness of $\sigma_2$
that one may choose the lift of $\tau_2$ to $\R$ such that
$\sigma_2(0)\in[0,\,2\pi)$. This is the convention used throughout.

\begin{definition}\cite[Def.~3.1]{vungoc0} \label{ffinv}
  Let $S_i$ be the unique smooth function defined around
  $0\in\R^2$ such that
  \begin{eqnarray} \label{xx} \left\{ \begin{array} {rl}
        \DD{}S_i=\sigma \,\,\\
        S_i(0)=0
      \end{array} \right.,
  \end{eqnarray}
  where $\sigma$ is the one\--form given by (\ref{sigma}).
  The Taylor expansion of $S$ at $(0,\,0)$ is denoted by $(S)^\infty$.
  We say that $(S_i)^\infty$ is the {\em Taylor series invariant
  of $(M,\, \omega,\, (J,\,H))$ at the focus\--focus point $c_i$},
  where $i \in \{1,\,\ldots,\,m_f\}$.
\end{definition}

The Taylor expansion $(S)^{\infty}$ is a formal power series in two
variables with vanishing constant term.  

\begin{lemma} \label{ii)}
Let $(M_1, \, \omega_1, (J_1,\,H_1)),\, (M_2, \, \omega_2, (J_2,\,H_2))$ be
isomorphic $4$\--dimensional
semitoric integrable systems and let $((S^j_i)^{\infty})_{i=1}^{m^i_f}$ be the tuple of Taylor
series invariants at the focus\--focus critical points of $(M_j, \, \omega_j, (J_j,\,H_j))$, where
$j \in \{1,\,2\}$. Then 
the tuple $((S^1_i)^{\infty})_{i=1}^{m^1_f}$ is equal to the tuple $((S^2_i)^{\infty})_{i=1}^{m^2_f}$.
\end{lemma}

This result was proven
in \cite{vungoc}.

\section{Combinatorial invariants of a semitoric system}
\label{sec:polytope}

The Atiyah\--Guillemin\--Sternberg and Delzant
theorems tell us that a lot of the information of some completely integrable
systems {\em coming from Hamiltonian torus actions} is encoded
combinatorially by polytopes.

\begin{figure}[htb]
\begin{center}
\includegraphics{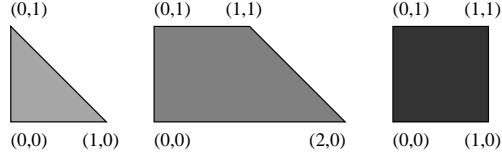}
\caption{Momentum polytope of $\mathbb{CP}^2$ (left),
a Hirzebruch surface (center) and $(\mathbb{CP}^1)^2$ (right), all
of which determine the isomorphism type of the manifold.}
\end{center}
\label{AF4}
\end{figure}

Although $4$\--dimensional semitoric systems are
not induced by torus actions, some of the information of
the system may be combinatorially encoded by
 a certain equivalence class of rational convex polygons
endowed with a collection of vertical weighted lines.
This is in fact a way of encoding the \emph{affine
 structure} induced by the integrable system.
 Throughout this section
 $(M,\, \omega,\, (J,\,H))$ is a $4$\--dimensional semitoric integrable system
 with $m_f$ isolated focus\--focus singular values.

\subsection{Affine Structures}
Recall that a map $X \subset \mathbb{R}^m \to Y \subset \mathbb{R}^m$
is \emph{integral\--affine on $X$} if it is of the form $
A_{ij}(\cdot)+b_{ij}, $ where $A_{ij} \in \textup{GL}(m,\,
\mathbb{Z})$ and $b_{ij} \in \mathbb{R}^m$.

An \emph{integral\--affine smooth $m$\--dimensional manifold} is a
smooth $m$\--dimensional manifold $X$ for which the coordinate changes
are integral\--affine, i.e. if $\varphi_i \colon U_i \subset
\mathbb{R}^m \to X$ are the charts associated to $X$, for all $i,\, j$
we have that $\varphi_i \circ \varphi_j^{-1}$, whereever defined, is
an integral affine map. We allow manifolds with boundary and corners,
in which case the charts take their values in $[0,+\infty)^k\times
\R^{m-k}$ for some integer $k\in\{0,\dots,m\}$.

A map $f \colon X \to Y$ between integral affine manifolds is
\emph{integral\--affine} if for each point $x \in X$ there are charts
$\varphi_x \colon U_x \to X$ around $x$ and $\psi_y \colon V_y \to Y$
around $y:=f(x)$ such that $\psi_y^{-1} \circ f \circ \varphi_x$ is
integral\--affine.

Any Lagrangian fibration $F:M\to B$ naturally defines an
integral-affine structure on the base $B$. This affine structure can
be characterized by the following fact~: a local diffeomorphism
$g:(B,\,b)\to(\R^n,\,0)$ is integral-affine if and only if the Hamiltonian
flows of the $n$ coordinate functions of $g\circ F$ are periodic of
primitive period equal to $2\pi$. Thus, an integrable system with
momentum map $F=(J,\,H)$ defines an integral-affine structure on the set
$B_r$ of regular values of $F$. In our case, this structure will in
fact extend to the boundary of $B_r$. Although $B_r$ is a subset of
$\R^2$, the integral\--affine structure of $B_r$ is in general different
from the induced canonical integral-affine structure of $\R^2$.

The integral\--affine structure of $B_r$ encodes much of the topology of
the integrable system (see~\cite{ntz1}) but, as we will see, is far
from encoding all its symplectic geometry.

\subsection{Generalized toric map}

We start with two definitions that we shall need.
Let $\mathfrak{I}$ be the subgroup of the affine group
$\op{Aff}(2,\,\Z)$ in dimension 2 of those transformations which leave
a vertical line invariant, or equivalently, an element of
$\mathfrak{I}$ is a vertical translation composed with a matrix $T^k$, where $k \in \Z$ and
\begin{eqnarray}
  T^k:=\left(
    \begin{array}{cc}
      1 & 0\\ k & 1
    \end{array}
  \right) \in \op{GL}(2,\, \Z).
  \label{equ:Tk}
\end{eqnarray}
Let $\ell\subset\R^2$ be a vertical line in the plane, not
necessarily through the origin, which splits it into two
half\--spaces, and let $n\in\Z$. Fix an origin in $\ell$.  Let
$\op{t}^n_{\ell} \colon \R^2 \to \R^2$ be the identity on the
left half\--space, and $T^n$ on the right half\--space. By definition
$\op{t}^n_{\ell}$ is piecewise affine.
Let $\ell_i$ be a vertical line through the focus\--focus
value $c_i=(x_i,\,y_i)$, where $1 \le i \le m_f$, and for any tuple $\vec
n:=(n_1,\,\dots,\,n_{m_f})\in\Z^{m_f}$ we set
\begin{eqnarray} \label{tn:eq}
\op{t}_{\vec n}:=\op{t}^{n_1}_{\ell_1}\circ\, \cdots\, \circ
\op{t}^{n_{m_f}}_{\ell_{m_f}}.
\end{eqnarray}
The map $t_{\vec n}$ is piecewise affine.

In \cite[Th.~3.8]{vungoc}
V\~u Ng\d oc describes how to associate to $(M,\, \omega, \,
F=(J, \, H))$ a rational convex polygon: the image
of a certain almost everywhere integral\--affine homeomorphism
$f \colon F(M) \subset  \R^2 \to
  \Delta \subset \R^2$.
Here, $B:=F(M)$ is equipped with the natural
integral\--affine structure induced by the system, while $\R^2$ on the
right hand-side is endowed with its canonical integral-affine
structure.  

Given a sign $\epsilon_i\in\{-1,+1\}$, let $\ell_i^{\epsilon_i}\subset\ell_i$ be
the vertical half line starting at $c_i$ at extending in the direction
of $\epsilon_i$~: upwards if $\epsilon_i=1$, downwards if
$\epsilon_i=-1$. Let
\[
\ell^{\vec\epsilon}:= \bigcup_{i=1}^{m_f}\ell_i^{\epsilon_i}.
\]
\begin{theorem}[Th.~3.8 in \cite{vungoc}] \label{keytheorem}
  \label{thm3.8advances}
  For $\vec\epsilon\in\{-1,+1\}^{m_f}$ there is a
  homeomorphism $f = f_\epsilon\colon B \to \R^2$
  such that
  \begin{enumerate}
  \item[(1)] $f|_{(B\setminus \ell^{\vec\epsilon})}$ is a
    diffeomorphism into its image $\Delta:=f(B)$.
  \item[(2)] $f|_{(B_r\setminus \ell^{\vec\epsilon})}$ is affine: it
    sends the integral affine structure of $B_r$ to the standard
    structure of $\R^2$.
  \item[(3)] $f$ preserves $J$: i.e. $f(x,\,y)=(x,\,f^{(2)}(x,\,y))$.
  \item[(4)]
  For any $i\in \{1,\,\ldots,\,m_f\}$ and any $c \in \ell_i^{\epsilon_i}\setminus\{c_i\}$ there is 
  an open ball $D$ around $c$ such that 
 $f|_{(B_r\setminus l^{\vec\epsilon})}$
has a smooth extension on each domain 
$\{(x,\, y) \in D \, \mid \, \le x_i\}$ and $\{(x,\,y) \in D\, \mid \,  x \ge x_i\}$. One 
has the formula: 
       $$\lim_{\substack{(x,y)\to c\\x<x_i}}\op{d}f(x,\,y) =
      T^{k(c)}\lim_{\substack{(x,y)\to
          c\\x>x_i}}\op{d}f(x,\,y),
      $$
    where $k(c)$ is the multiplicity of $c$.
  \item[(5)] The image of $f$ is a rational convex polygon.
  \end{enumerate}
  Such an $f$ is unique modulo a left composition by a transformation
  in $\mathfrak{I}$. 
\end{theorem}

In order to arrive at the rational convex polygon $\Delta$ in the proof of Theorem \ref{keytheorem}
one cuts the image $(J,\,H)(M) \subset \R^2$, which is in general not
convex, along each of the vertical lines
$\ell_i$, $i \in \{1,\,\ldots,\,m_f\}$. One must make a choice
of ``cut direction'' for each vertical line $\ell_i$, that is to say that one has to choose
whether to cut the set $(J,\,H)(M)$ along the half\--vertical\--line $\ell_i^{+1}$ which starts at $c_i$ going upwards,
or along the half\--vertical\--line $\ell_i^{-1}$ which starts at $c_i$ going downwards.
Precisely, 
the definitions of  $f$ and $\Delta$ 
in Theorem \ref{keytheorem} depend on two
choices in the proof~:
\begin{itemize}
\item[(a)] {\em an initial set of action variables $f_0$ of the form $(J,\,K)$} near a
  regular Liouville torus in \cite[Step 2, pf. of
  Th.~3.8]{vungoc}.
  If we choose $f_1$ instead of $f_0$
  we get a polytope $\Delta'$
  by left composition with an element of $\mathfrak{I}$.  Similarly
  instead instead of $f$ we obtain $f$ composed on the left with an
  element of $\mathfrak{I}$;
\item[(b)]  {\em a tuple $\vec{\epsilon}$ of $1$ and $-1$}.
If we choose $\vec{\epsilon'}$ instead of $\vec{\epsilon}$
  we get $ \Delta'=\op{t}_{\vec{u}}(\Delta) $ with
  $u_i=(\epsilon_i-\epsilon'_i)/2$, by \cite[Prop. 4.1,
  expr. (11)]{vungoc}.  Similarly instead of $f$ we obtain
  $f'=\op{t}_{\vec{u}} \circ f$.
\end{itemize}

\begin{definition} \label{toricmap} 
Let $(M,\, \omega,\, (J,\,H))$ be a semitoric integrable system and let $f$ a
choice of homeomorphism as in Theorem \ref{keytheorem}.
  We say that: 
  \begin{itemize}
  \item[(i)]
the map  $f\circ(J,\,H)$ is a \emph{generalized toric momentum map
    for $(M,\, \omega,\, (J, \, H))$};
 \item[(ii)]
the rational convex polygon
 $\Delta:=f\big((J,\,H)(M)\big)$ is a a \emph{generalized toric momentum polygon
    for $(M,\, \omega,\, (J, \, H))$}.
    \end{itemize}
    \end{definition}

For simplicity sometimes we omit the word ``generalized'' in
Definition \ref{toricmap}.

\subsection{Semitoric polygon invariant}
Let $\op{Polyg}(\R^2)$ be the space of rational
convex polygons in $\R^2$. Let $\op{Vert}(\R^2)$
be the set of vertical lines in $\R^2$, i.e.
$$\op{Vert}(\R^2)=\Big\{\ell_{\lambda}:=\{(x,\,y) \in \R^2 \, | \, x=\lambda\} \,\, | \,\, \lambda \in \R\Big\}.$$
\begin{definition} \label{semitoric-pol:def}
A {\em weighted polygon of complexity $s$} is a triple of the form
$$
\Delta_{\scriptop{weight}}=\Big(\Delta,\,
(\ell_{\lambda_j})_{j=1}^s,\, (\epsilon_j)_{j=1}^s\Big)
$$
where $s$ is a non\--negative integer and:
\begin{itemize}
\item
$\Delta \in \op{Polyg}(\R^2)$;
\item
$\ell_{\lambda_j} \in \op{Vert}(\R^2)$ for every $j \in \{1,\ldots,s\}$;
\item
$\epsilon_j \in \{-1,\,1\}$ for every $j \in \{1,\ldots,s\}$;
\item
$
\op{min}_{s \in \Delta}\pi_1(s)<\lambda_1<\ldots<\lambda_s<
\op{max}_{s \in \Delta}\pi_1(s),
$
where $\pi_1 \colon \R^2 \to \R$ is the canonical projection $\pi_1(x,\,y)=x$.
\end{itemize}
We denote by $\mathcal{W}\op{Polyg}_s(\R^2)$ the space
of all weighted polygons of complexity $s$.
\end{definition}

\begin{figure}[htb]
\begin{center}
\includegraphics{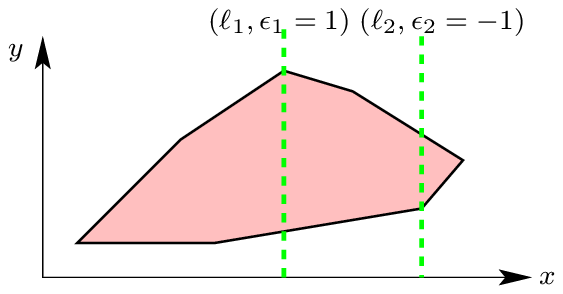}
\caption{A weighted polygon of complexity $2$.}
\end{center}
\label{AF3}
\end{figure}

For any $s \in \N$, let 
\begin{eqnarray} \label{Gs:eq}
G_s:=\{-1,\,+1\}^s
\end{eqnarray}
 and let 
 \begin{eqnarray} \label{calG:eq}
 \mathcal{G}:=\{T^k\,\, | \,\, k \in \Z\}, 
 \end{eqnarray}
 where $T$ is the $2$ by $2$ matrix (\ref{equ:Tk}).
Consider the action
of the product group $G _s\times \mathcal{G}$
on the space  $\mathcal{W}\op{Polyg}_s(\R^2)$:
the product
$$
((\epsilon'_j)_{j=1}^s,\,T^k) 
\cdot \Big( \Delta,\,(\ell_{\lambda_j})_{j=1}^s,\, (\epsilon_j)_{j=1}^s\Big)
$$
is defined to be
\begin{eqnarray} \label{productaction}
\Big(\op{t}_{\vec u}(T^k(\Delta)),\,(\ell_{\lambda_j})_{j=1}^s,\, (\epsilon'_j\,\epsilon_j)_{j=1}^s\Big),
\end{eqnarray}
where $\vec u=((\epsilon_i-\epsilon'_i)/2)_{i=1}^s$,
and $\op{t}_{\vec{u}}$ is a map of the form (\ref{tn:eq}).

\begin{definition} \label{generalizedpolytope}
Let $\Delta$ be a rational convex polygon 
obtained from the momentum image $(J,\,H)(M)$
according to the proof of Theorem \ref{keytheorem}
by cutting along the vertical half-lines $\ell_1^{\epsilon_1},
\ldots,\ell_{m_f}^{\epsilon_{m_f}}$.
The {\em semitoric polygon invariant} of $(M,\, \omega,\, (J,\,H))$
is the $(G_{m_f} \times \mathcal{G})$\--orbit 
\begin{eqnarray} \label{polin}
(G_{m_f} \times \mathcal{G})\cdot \Big(\Delta,\, (\ell_j)_{j=1}^{m_f},\,
(\epsilon_j)_{j=1}^{m_f}\Big) \in \mathcal{W}\op{Polyg}_{m_f}(\R^2)/(G_{m_f} \times \mathcal{G}),
\end{eqnarray}
where $\mathcal{W}\op{Polyg}_{m_f}(\R^2)$ is as in 
Definition \ref{semitoric-pol:def}
and the action of $G_{m_f} \times \mathcal{G}$ on $\mathcal{W}\op{Polyg}_{m_f}(\R^2)$
is given by  (\ref{productaction}).
\end{definition}

It follows now from Theorem~\ref{thm3.8advances} that the semitoric polygon invariant
does not depend on the isomorphism class of the system.
\begin{lemma} \label{iii)}
Let $(M_1, \, \omega_1, (J_1,\,H_1)),\, (M_2, \, \omega_2, (J_2,\,H_2))$ be
isomorphic $4$\--dimensional semitoric integrable systems. Then the semitoric polygon invariant
of  $(M_1, \, \omega_1, (J_1,\,H_1))$ is equal to the semitoric 
polygon invariant of $(M_2, \, \omega_2, (J_2,\,H_2))$.
\end{lemma}

\section{Geometric invariants of a semitoric system}
\label{sec:geometric}
 
The invariants we have described so far are not
enough to determine whether two $4$\--dimensional
semitoric systems are isomorphic.  In this
section we introduce two global geometric invariants,
which encode a mixture of information about 
local and global behavior.
Throughout,
$(M, \, \omega,\, (J,\,H))$ is a $4$\--dimensional semitoric integrable system
with $m_f$ isolated focus\--focus singular values.

\subsection{The Volume Invariant}
\label{volumesection}

The invariant we introduce next is easy to define 
using the combinatorial ingredients we have by now introduced.
Consider a focus\--focus critical point $m_i$ whose image by $(J,\,H)$
is $c_i$ for $i \in \{1,\, \ldots,\, m_f\}$, and let $\Delta$ be a rational convex polygon
corresponding to the system $(M,\, \omega,\, (J,\,H))$,
c.f. Definition \ref{generalizedpolytope}.

\begin{lemma} \label{height:lemma}
If $\mu$ is a toric momentum map for the semitoric system 
$(M, \, \omega,\,(J, \, H))$ corresponding to $\Delta$, c.f. Definition \ref{toricmap}, 
then the image $\mu(m_i)$, where $i \in \{1,\,\ldots,\,m_f\}$,  is a point lying in the
interior of the polygon $\Delta$, along the line $\ell_i$.
  The vertical distance 
  \begin{eqnarray} \label{height:eq}
  h_i:=\mu(m_i)-\min_{s \in \ell_i \cap \Delta} \pi_2(s)>0
  \end{eqnarray}
  between $\mu(m_i)$ and the point of
  intersection of $\ell_i$ with the image polytope with
  lowest $y$\--coordinate, is independent of the choice of 
  momentum map $\mu$. Here
  $\pi_2 \colon \R^2 \to \R$ is the canonical projection
  $\pi_2(x,\,y)=y$.
   \end{lemma}


 Lemma \ref{height:lemma} follows  from the fact that two different toric
 momentum maps only differ by piecewise affine transformations, which
 all act on any fixed vertical line as translations.

\begin{definition}
We say that  the vertical distance (\ref{height:eq}) bewteen $\mu(m_i)$ and the point of
  intersection of $\ell_i$ with the image polytope that has the
  lowest $y$\--coordinate is {\em the height of the focus\--focus
  critical value} $c_i$, where $i \in \{1,\,\ldots,\,m_f\}$.
\end{definition}

\begin{remark}
 One can give a geometrical meaning to the height of the focus-focus
  critical values. Let $Y_i=J^{-1}(c_i)$. This singular manifold
  splits into two parts, $Y^+_i$ and $Y^-_i$ defined as
  $Y_i\cap\{H>H(m_i)\}$ and $Y_i\cap\{H<H(m_i)\}$, respectively. 
  The height of the focus-focus critical value $c_i$ is simply the
  Liouville volume of $Y^-_i$.
\end{remark}

Since isomorphic systems share the same set of momentum polygons, we
have the following result.
\begin{lemma} \label{iv)}
Let $(M_1, \, \omega_1, (J_1,\,H_1)),\, (M_2, \, \omega_2, (J_2,\,H_2))$ be
isomorphic $4$\--dimensional
semitoric integrable systems and let $(h^j_i)_{i=1}^{m^i_f}$ be the tuple
of heights of focus\--focus critical values of of $(M_j, \, \omega_j, (J_j,\,H_j))$, $j
\in \{1,\,2\}$. Then 
the tuple $(h^1_i)_{i=1}^{m^1_f}$ is equal to the tuple $(h^2_i)_{i=1}^{m^2_f}$.
\end{lemma}

The volume invariant is very easy to compute from a weighted polygon, and hence it
is a quick way to rule out that two semitoric integrable systems are not
isomorphic.

\subsection{The Twisting\--Index Invariant}
\label{indexsection}

For clarity,
we divide the construction of the twisting index invariant into five steps.
Let 
\begin{eqnarray} \label{pol:eq}
\Delta_{\scriptop{weight}}:=\Big(\Delta,\, (\ell_j)_{j=1}^{m_f},\,
(\epsilon_j)_{j=1}^{m_f}\Big) \in \mathcal{W}\op{Polyg}_{m_f}(\R^2),
\end{eqnarray}
be a weighted polygon as in 
expression (\ref{polin}), representing the orbit given
by the semitoric polygon invariant of the system
$(M,\, \omega,\, (J,\,H))$, c.f. Definition \ref{generalizedpolytope},
where recall that the polygon $\Delta$
is obtained from the momentum image $(J,\,H)(M)$
according to the proof of Theorem \ref{keytheorem}
by cutting along the vertical lines $\ell_1,
\ldots,\ell_{m_f}$ in the direction of
$\epsilon_1,\ldots,\epsilon_{m_f}$, i.e. upwards
if $\epsilon_i$ is $+1$ and downwards otherwise.
Write $F=(J,\,H)$, $c_1,\ldots,c_{m_f}$ for 
the focus\--focus critical values. 

In the first three steps we define for each
$i \in \{1,\,\dots,\,m_f\}$, an integer $k_i$ that we shall call the \emph{twisting
  index of} the focus\--focus value $c_i$, on which
  we built to construct the actual invariant associated to
$(M,\, \omega,\, (J,\,H))$ in Step 5.
\\
\\
{\em Step 1: an application of Eliasson's theorem}. 
Let $(e_1,\,e_2)$ be the canonical basis of $\R^2$.  Let
$\ell=\ell_i^{+1}\subset\R^2$ be the vertical half-line starting at $c_i$ and
pointing in the direction of $\epsilon_i\, e_2$.

Let us apply Eliasson's theorem in a small neighbourhood $W=W_i$ of
the focus-focus critical point $m_i=F^{-1}(c_i)$~: there exists a
local symplectomorphism $\phi:(\R^4,\,0)\to W$, and a local
diffeomorphism $g$ of $(\R^2,\,0)$ such that
\begin{equation}
  F \circ \phi= g\circ q,
  \label{equ:eliasson} 
\end{equation}
where $q$ is the quadratic momentum map given
by~\eqref{equ:cartan}. Since the second component, $q_2\circ\phi^{-1}$
has a $2\pi$\--periodic Hamiltonian flow, it must be equal to $J$ in
$W$, up to a sign.  Composing if necessary $\phi$ by the canonical
transformation $(x,\xi)\mapsto(-x,-\xi)$, one can always assume that
$q_2=J\circ\phi$ in $W$. This means that $g$ is of the form
\begin{equation}
  g(q_1,\,q_2)=(q_2,\,g_2(q_1,\,q_2)).
  \label{equ:g} 
\end{equation}

Moreover, upon composing $\phi$ by the canonical transformation
$(x,\,y,\,\xi,\,\eta)\mapsto(-\xi,\,-\eta,\,x,\,y)$, which changes $(q_1,\,q_2)$
into $(-q_1,\,q_2)$, one can always assume that
\begin{equation}
  \frac{\partial g_2}{\partial q_1}(0)>0.
  \label{equ:sign}
\end{equation}
In particular, near the origin $\ell$ is transformed by $g^{-1}$ into
the positive real axis if $\epsilon_i=1$, or the negative real axis if
$\epsilon_i=-1$.
\\
\\
{\em Step 2: the smooth vector field $\mathcal{X}_p$}. 
Let us now fix the origin of angular polar coordinates in $\R^2$ on
the \emph{positive} real axis. Let $V=F(W)$ and define
$\tilde{F}=(H_1,\,H_2)=g^{-1}\circ F$ on $F^{-1}(V)$ (notice that
$H_2=J$). Now recall from Section~\ref{focussection} that near any
regular torus there exists a Hamiltonian vector field $\mathcal{X}_p$,
whose flow is $2\pi$\--periodic, defined by
\begin{equation}
  2\pi\mathcal{X}_p =
  (\tau_1\circ\tilde{F})\mathcal{X}_{H_1}+(\tau_2\circ\tilde{F})\mathcal{X}_J,
  \label{equ:Xp} 
\end{equation}
where $\tau_1$ and $\tau_2$ are functions on $\R^2\setminus\{0\}$
satisfying~\eqref{equ:sigma}, with $\sigma_1(0)>0$. In fact $\tau_2$
is multivalued, but we determine it completely in polar coordinates
with angle in $[0,\,2\pi)$ by requiring continuity in the angle variable
and $\sigma_2(0)\in[0,\,2\pi)$. In case $\epsilon_i=1$, this defines
$\mathcal{X}_p$ as a smooth vector field on
$F^{-1}(V\setminus\ell)$. In case $\epsilon_i=-1$ we keep the same
$\tau_2$\--value on the negative real axis, but extend it by continuity
in the angular interval $[\pi,\,3\pi)$. In this way $\mathcal{X}_p$ is
again a smooth vector field on $F^{-1}(V\setminus\ell)$.
\\
\\
{\em Step 3: twisting index of a weighted polygon at a focus\--focus singularity}. 
Let $\mu$ be the generalized toric momentum map, c.f. Definition
\ref{toricmap}, associated to the polygon $\Delta$. On
$F^{-1}(V\setminus \ell$), $\mu$ is smooth, and its components
$(\mu_1,\,\mu_2)=(J,\,\mu_2)$ are smooth Hamiltonians, whose vector fields
$(\mathcal{X}_J,\mathcal{X}_{\mu_2})$ are tangent to the foliation,
have a $2\pi$-periodic flow, and are \emph{a.e.}  independent. Since
the couple $(\mathcal{X}_J,\mathcal{X}_p)$ shares the same properties,
there must be a matrix $A\in\textup{GL}(2,\Z)$ such that
$(\mathcal{X}_J,\mathcal{X}_{\mu_2}) = A
(\mathcal{X}_J,\mathcal{X}_p)$. This is equivalent to saying that
there exists an integer $k_i\in\Z$ such that
\begin{eqnarray} \label{ki:eq}
\mathcal{X}_{\mu_2} = k_i\mathcal{X}_{J} + \mathcal{X}_p.
\end{eqnarray}

\begin{prop}
For a fixed weighted polygon $\Delta_{\scriptop{weight}}$ as in equation (\ref{pol:eq}), 
the integer  $k_i$ in (\ref{ki:eq})
is well defined for each $i \in \{1,\,\ldots,\,m_f\}$, i.e. it does not depend on 
\begin{itemize}
\item[(a)] the choice of the periodic
  Hamiltonian $\mathcal{X}_p$;
\item[(b)]
  the transformations involved in Eliasson's normal
  form~(\ref{equ:eliasson}), with the sign constraints~(\ref{equ:g})
  and~(\ref{equ:sign}).
\end{itemize}
\end{prop}

\begin{proof}
  It follows from Lemma~4.1 in~\cite{vungoc0} that changing the
  transformations involved in Eliasson's normal form \cite{eliasson-these} can only modify
  $g_2$ --- and hence $H_1$ --- by a flat term. Suppose
  $\mathcal{X}_p'$ is another admissible choice for a Hamiltonian of
  the form
  \[
  2\pi\mathcal{X}_p' =
  (\tau_1'\circ\tilde{F}')\mathcal{X}_{H_1'}+(\tau_2'\circ\tilde{F}')\mathcal{X}_J.
  \]
  Since $\mathcal{X}_p'$ has a $2\pi$-periodic flow, there must be
  coprime integers $a,\,b$ in $\Z$ such that
  \begin{equation}
    \mathcal{X}_p' = a\mathcal{X}_p + b\mathcal{X}_J.
    \label{equ:period}
  \end{equation}
  Inserting~(\ref{equ:Xp}), we see that there exist functions $Z_1$
  and $Z_2$ that vanish at all orders at the origin such that
  $$
  2\pi\mathcal{X}_p' = (a\tau_1\circ\tilde{F}+Z_1)\mathcal{X}_{H_1'}
  + (a\tau_2\circ\tilde{F}+2\pi b+Z_2)\mathcal{X}_J. 
  $$
  From this we see
  that, up to a flat function, $\tau_1'=a\tau_1$ and
  $\tau_2'=a\tau_2+2\pi b$. Because of the logarithmic asymptotics
  required in~(\ref{equ:sigma}), the first equation requires
  $a=1$. But then, the second equation with the restriction that both
  $\sigma_2(0)$ and $\sigma_2'(0)$ must be in $[0,2\pi)$ implies that
  $b=0$. Recalling~(\ref{equ:period}) we obtain
  $\mathcal{X}_p'=\mathcal{X}_p$, which shows that $k_i$ is indeed
  well-defined.
\end{proof}

\begin{definition} \label{ki:def}
Let $\Delta_{\scriptop{weight}}$ be a fixed weighted 
polygon as in (\ref{pol:eq}).
For each $i \in \{1,\,\ldots,\,m_f\}$, 
the integer $k_i$ defined in equation (\ref{ki:eq}) is called the \emph{twisting index of 
  $\Delta_{\scriptop{weight}}$ at the focus\--focus critical value $c_i$.}
\end{definition}
The integer $k_i$ in Definition \ref{ki:def}
is still not the relevant object that we intend to associate to the semitoric
system, but we shall build on its definition to construct the actual invariant.
\\
\\
{\em Step 4: the privileged momentum map}.
We explain how there is a reasonable way to ``choose'' a momentum map
for $(M,\, \omega,\, (J,\,H))$.

\begin{lemma}
  \label{rema:priviledged}
  There exists a unique smooth function $H_p$ on $F^{-1}(V\setminus
  \ell)$ the  Hamiltonian vector field of which is $\mathcal{X}_p$ and such
  that $\lim_{m\to m_i}H_p=0.$ 
  \end{lemma}
  
  \begin{proof}
  Near a regular torus
  $\mathcal{X}_p$ is a Hamiltonian vector field of a function of the
  form $f(H_1,\,J)$, and by construction 
  $
  \partial_i f = \tau_i/2\pi,
  $
  $i \in \{1,\,2\}$.  Therefore, using~(\ref{equ:sigma}) we can check that 
  $
  2\pi
  f(z)=S(c)-\Re(z\ln z -z) + \textup{Const}, 
  $
  where $S$ is smooth at
  the origin, which shows that $f$ has a limit as $z\in
  \R^2\setminus([0,\,\infty)\times\{0\})$ tends to the origin.
  In fact,
  $f$ has a continuous extension to $\R^2$, entailing that $H_p$
  extends to a continuous function on $F^{-1}(V)$.
\end{proof}

  \begin{definition} \label{def:priviledged}
  Let $(M,\, \omega,\, (J,\,H))$ be a $4$\--dimensional semitoric integrable system,
  and let $H_p$ be the unique smooth function defined in
  Lemma \ref{rema:priviledged}.
  We say that the toric momentum map 
  $\nu:=(J,\,H_p)$
  is {\em the
  privileged momentum map for $(J,\,H)$} around the focus\--focus value $c_i$,
  for each $i \in \{1,\,\ldots,\,m_f\}$.
  \end{definition}
  
\begin{remark}  
The map $\nu$  in Definition \ref{def:priviledged} depends
on the cut $\ell$, that is to say, on the sign
  $\epsilon_i$. Moreover, we have the following.
  \begin{itemize}
  \item[(a)]
  If $k_i$ is the twisting index of $c_i$, one has
  \begin{equation}
    \mu = T^{k_i} \nu \qquad \text{ on } F^{-1}(V).
    \label{equ:twist}
  \end{equation}
\item[(b)] If we transform the polygon $\Delta$ by a global affine
  transformation in $T^r\in\mathfrak{I}$ this has no effect on the
  privileged momentum map $\nu$, whereas it changes $\mu$ into $T^r
  \mu$.  
  \end{itemize}
    From the characterization~(\ref{equ:twist}), it follows that
  all the twisting indices $k_i$ are replaced by $k_i+r$.
\end{remark}
With this preparation we are now ready to define the twisting index invariant.
\\
\\
{\em Step 5: the twisting index invariant}. 
We give the definition of the twisting index invariant as an equivalence
class of weighted polygons pondered by a collection of integers.

\begin{prop} \label{k_i:lemma}
If two weighted polygons $\Delta_{\scriptop{weight}}$ and $\Delta'_{\scriptop{weight}}$ lie in the same
$G_{m_f}$\--orbit, for the $G_{m_f}$\--action induced by (\ref{productaction}), then
the twisting indexes $k_i,\,k'_i$ associated to $\Delta_{\scriptop{weight}}$
and $\Delta'_{\scriptop{weight}}$ at their respective focus\--focus
critical values $c_i,\,c'_i$ are equal,
for each $i \in \{1,\,\ldots,\,m_f\}$. 
\end{prop}

\begin{proof}
For $\epsilon_i=\pm 1$, we denote by
$\mu_\pm$ and $\nu_\pm$, as in equation~(\ref{equ:twist}) above, the
generalized toric momentum map and the privileged momentum map, c.f.
Definition \ref{def:priviledged} at
$c_i$. 

With the notations of Section~\ref{sec:polytope}, we have
$$
\mu_-=\op{t}_{\ell_i}\mu_+. 
$$

On the other hand, from the definition
of $\tau_2$ in each case, we see that
$\mathcal{X}_{p,-}=\mathcal{X}_{p,+}$ on the left-hand side of $\ell$
(that is to say, $J<0$), while
$$
\mathcal{X}_{p,-}=\mathcal{X}_{p,+}+2\pi\mathcal{X}_J
$$
on the right hand side ($J>0$). This means that
$\nu_-=\op{t}_{\ell_i}\nu_+$. From the characterization of the
twisting index by~(\ref{equ:twist}), using that $\op{t}_{\ell_i}$
commutes with $T$, we see that $k_{i,+}=k_{i,-}$.
\end{proof}

Recall the groups $G_s$ and $\mathcal{G}$ given by (\ref{Gs:eq}) and
(\ref{calG:eq}) respectively,
and the action of $G_s \times \mathcal{G}$ on $\mathcal{W}\op{Polyg}_s(\R^2)$,
c.f. Definition \ref{productaction}.
Consider the action
of the product group $G_s \times \mathcal{G}$
on the space  $\mathcal{W}\op{Polyg}_s(\R^2) \times \Z^{m_s} $:
the product
$$
((\epsilon'_j)_{j=1}^s),\,T^k) 
\star \Big( \Delta,\,(\ell_{\lambda_j})_{j=1}^s,\, (\epsilon_j)_{j=1}^s,\, (k_i)_{i=1}^{s}\Big)
$$
is defined to be
\begin{eqnarray} \label{newproductaction}
\Big(\op{t}_{\vec{u}}(T^k(\Delta)),\,(\ell_{\lambda_j})_{j=1}^s,\, (\epsilon'_j\,\epsilon_j)_{j=1}^s,
\, (k_i+k)_{i=1}^{s}\Big). 
\end{eqnarray}
where $\vec u=(\epsilon_i-\epsilon'_i)/2)_{i=1}^s$. Here
$T$ is the $2$ by $2$ matrix (\ref{equ:Tk}) and $\op{t}_{\vec{u}}$ is of the form
(\ref{tn:eq}).

\begin{definition} \label{indexinv}
The {\em twisting\--index invariant} of $(M,\, \omega,\, (J,\,H))$
is the $(G_{m_f} \times \mathcal{G})$\--orbit of weighted polygon pondered
by twisting indexes at the focus\--focus singularities of the system given by
\begin{eqnarray} 
(G_{m_f} \times \mathcal{G})\star \Big(\Delta,\, (\ell_j)_{j=1}^{m_f},\,
(\epsilon_j)_{j=1}^{m_f},\, (k_i)_{i=1}^{m_f}\Big) \in (\mathcal{W}\op{Polyg}_{m_f}(\R^2) \times \Z^{m_f})/(G_{m_f} \times \mathcal{G}),
\end{eqnarray}
where $\mathcal{W}\op{Polyg}_{m_f}(\R^2)$ is defined in Definition \ref{semitoric-pol:def}
and the action of $G_{m_f} \times \mathcal{G}$ on $\mathcal{W}\op{Polyg}_{m_f}(\R^2) \times \Z^{m_f}$
is given by  (\ref{newproductaction}).
\end{definition}

Here again, our definition is invariant under isomorphism.
\begin{lemma} \label{v)}
Let $(M_1, \, \omega_1, (J_1,\,H_1)),\, (M_2, \, \omega_2, (J_2,\,H_2))$ be isomorphic
$4$\--dimensional
semitoric integrable systems. Then
their corresponding twisting\--index invariants are equal.
\end{lemma}

\begin{remark}
  We would like to emphasize again 
  that the twisting index is
  \emph{not} a semiglobal invariant of the singular fibration in a
  neighbourhood of the focus-focus fibre. Such semiglobal fibrations
  are completely classified in~\cite{vungoc0}, and the twisting index
  does not play any role there. It is instead a \emph{global}
  invariant characterizing the way the fibers near a particular
  focus\--focus point stand with respect to the rest of the fibration.
\end{remark}

\section{Main Theorem: Statement}
\label{sec:main}

To each semitoric system we assign a list of invariants as above.

\begin{definition} \label{listofinvariants}
  Let $(M, \, \omega, \, (J, \, H))$ be a $4$\--dimensional semitoric integrable
  system.  The {\em list of invariants of  $(M, \, \omega,
    \, (J, \, H))$} consists of the following items.
  \begin{itemize}
  \item[(i)] The integer number $0 \le m_f<\infty$ of focus\--focus singular points, see
    Section \ref{mfsection}.
  \item[(ii)] The $m_{f}$\--tuple 
    $((S_i)^{\infty})_{i=1}^{m_{f}}$, where
    $(S_i)^{\infty}$ is the Taylor series of the
    $i^{\scriptop{th}}$ focus\--focus point, see Section
    \ref{focussection}.
  \item[(iii)] The semitoric polygon invariant: the
  $(G_{m_f} \times \mathcal{G})$\--orbit 
  $$
  (G_{m_f} \times \mathcal{G}) \cdot \Delta_{\scriptop{weight}}
 \,\, \in \,\, \mathcal{W}\op{Polyg}(\R^2)/ (G_{m_f} \times \mathcal{G})
  $$
  of the weighted polygon
  $\Delta_{\scriptop{weight}}:=\Big(\Delta,\, (\ell_j)_{j=1}^{m_f},\,
(\epsilon_j)_{j=1}^{m_f}\Big)$, c.f. Definition \ref{generalizedpolytope}.  
 
   \item[(iv)] The $m_f$\--tuple of integers
    $(h_i)_{i=1}^{m_{f}}$, where $h_i$ is the height of
    the $i^{\scriptop{th}}$ focus\--focus point, see
    Section \ref{volumesection}.
  \item[(v)] 
 The twisting\--index invariant: the
$(G_{m_f} \times \mathcal{G})$\--orbit
$$
(G_{m_f} \times \mathcal{G}) \star \Big(\Delta_{\scriptop{weight}},\, (k_i)_{i=1}^{m_f}\Big)  \,\,\in\,\,
(\mathcal{W}\op{Polyg}(\R^2) \times \Z^{m_f})/ (G_{m_f} \times \mathcal{G})
$$
of the weighted polygon
pondered by the twisting\--indexes
$\Big(\Delta_{\scriptop{weight}},\, (k_i)_{i=1}^{m_f}\Big)$,  
 c.f. Definition \ref{indexinv}.  
\end{itemize}
\end{definition}

In the above list invariant (v)
determines invariant (iii), so we could have ignored the latter.  We
have kept this list as it appears naturally in the construction of the
invariants. Indeed the definition of invariant (iii) is needed to
construct invariant (v). One may also argue that it is worthwhile for
practical purposes to list (iii), as it is easier to compute than (v)
and hence if two systems do not have the same invariant (iii) we know
they are not isomorphic without having to compute (v).

 Recall that if
$(M_1,\, \omega_1,\, (J_1,\,H_1))$ 
and
$(M_2,\, \omega_2,\, (J_2,\,H_2))$
are $4$\--dimensional
semitoric integrable systems,
we say that they
  are isomorphic  if there exists a symplectomorphism
$$
\varphi \colon M_1 \to M_2,\,\,\,\,\,\,\, \textup{such that}\,\,\,\,\,
\varphi^*(J_2,\,H_2)=(J_1,\,f(J_1,\,H_1)).
$$
for some smooth function $f$.
Our main theorem is the following.

\begin{theorem} \label{mainthm}
Two $4$\--dimensional semitoric integrable systems $(M_1, \, \omega_1, (J_1,\,H_1)),\, (M_2, \, \omega_2, (J_2,\,H_2))$ 
are
 isomorphic if and only if the list of
  invariants (i)\--(v), as in Definition \ref{listofinvariants}, of $(M_1, \, \omega_1, (J_1,\,H_1))$  is equal to the list of invariants (i)\--(v)
  of $(M_2, \, \omega_2, (J_2,\,H_2))$.
\end{theorem}

The proof of Theorem \ref{mainthm} is sufficiently involved that is better
organized in an independent section. In the proof we combine notable
results of several authors, in particular
Eliasson, Duistermaat, Dufour\--Molino, Liouville\--Mineur\--Arnold
and  V\~u Ng\d oc. We combine these results with new ideas
to construct explicitly an isomorphism between two semitoric integrable
systems that have the same invariants, in the spirit of  Delzant's proof 
\cite{delzant} for the case when the system defines a Hamiltonian
$2$\--torus action. Because in our context we have focus\--focus
singularities a number of delicate problems arise that one has to deal with
to construct such an isomorphism. 
As a matter of fact,
it is remarkable how the behavior of the system 
near a particular singularity
has a subtle global effect on the system.

\section{Proof of Main Theorem}
\label{sec:proof}

The left\--to\--right implication 
follows from putting together lemmas \ref{i)}, \ref{ii)}, \ref{iii)}, \ref{iv)} and \ref{v)}.
The proof of the right\--to\--left implication breaks into three steps.
 Let $F_1=(J_1,\,H_1)$ and let $F_2=(J_2,\,H_2)$. 
\begin{itemize}
\item
First, we reduce to a case where the images $F_1(M_1)$ and $F_2(M_2)$
are equal. 
\item
Second, we prove that this common image can be covered by
open sets $\Omega_\alpha$, above each of which $F_1$ and $F_2$ are
symplectically interwined. 
\item
The last step is to glue together these
local symplectomorphisms, in this way constructing a global
symplectomorphism $\phi:M_1\to M_2$ such that 
$F_1=F_2\circ\phi$.
\end{itemize}
{\em Step 1}: \emph{first reduction}.  The goal of this step is to
reduce to a particular case where the images $F_1(M_1)$ and $F_2(M_2)$
are equal. 
For simplicity we assume that the invariants of $F_1$ are indexed as in Definition
\ref{listofinvariants} with an additional upper index $1$, and similarly for $F_2$.

\begin{figure}[htbp]
 \begin{center}
    \includegraphics{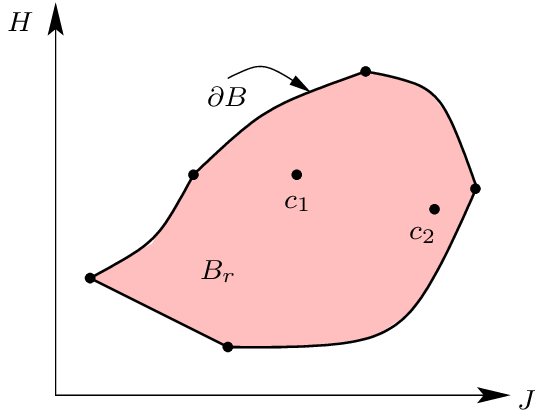}
    \caption{In Step 1 we prove that we can assume that the ``momentum'' images
    $F_1(M_1)$ and $F_2(M_2)$  are
    equal to the same curved polygon $B$. To emphasize this we index the axes
    as $J$ and $H$ without lower indexes.}
    \label{fig:image}
  \end{center}
\end{figure}

Because both systems have the same invariants (i), (iii) and (v),
we may choose a weighted polygon pondered by the twisting
indexes
$
\Big( \Delta_{\scriptop{weight}},\, (k_i)_{i=1}^{m_f}\Big),
$
where
$
\Delta_{\scriptop{weight}}=\Big(\Delta,\, (\ell_j)_{j=1}^{m_f},\,
(\epsilon_j)_{j=1}^{m_f} \Big),
$
and which is inside of the
$(\mathcal{W}\op{Polyg}(\R^2) \times \Z^{m^1_f})/(G_{m^1_f} \times \mathcal{G})=
(\mathcal{W}\op{Polyg}(\R^2) \times \Z^{m_f^2})/(G_{m^2_f} \times \mathcal{G})$\--orbit
of weighted polygons pondered by twisting indexes:
\begin{eqnarray} \label{for2:eq}
&&
(G_{m^1_f} \times \mathcal{G})\star \Big( \Delta^1_{\scriptop{weight}}\, (k^1_i)_{i=1}^{m^1_f}\Big)=
(G_{m^2_f} \times \mathcal{G})\star \Big(\Delta^2_{\scriptop{weight}},\, (k^2_i)_{i=1}^{m^2_f} \Big), 
\end{eqnarray}
where we are writing
$\Delta^i_{\scriptop{weight}}=(\Delta^i,\, (\ell^i_j)_{j=1}^{m^i_f},\,
(\epsilon^i_j)_{j=1}^{m^i_f}\Big)$, for $i \in \{1,\, 2\}$.

Let
$\mu_1,\,\mu_2$ respectively be associated toric momentum maps to $F_1$ and $F_2$
for the polygon $\Delta$ in Theorem \ref{keytheorem} and Definition \ref{toricmap}.
There are homeomorphisms $g_1,\,g_2 \colon
\Delta \to \Delta$ such that
$$
F_1=g_1 \circ \mu_1, \, \, \, \, F_2=g_2
  \circ \mu_2.
 $$
Consider the map $ h:=g_1 \circ g_2^{-1}$. We wish to replace $F_2$ by
$\widetilde{F}_2=h \circ F_2$. Then, obviously, 
$$\op{Image}(\widetilde{F}_2)=
g_1(\Delta)=\op{Image}({F}_1).
$$
In order for $\widetilde{F}_2$ to define a semi\--toric completely
integrable system isomorphic to $F_2$, we need to prove that
$h(x,\,y)=(x,\,f(x,y))$ for some smooth function $f$. In fact, it follows
from Theorem \ref{keytheorem} that $h$ has this form, but for some $f$ which is \emph{a
  priori} not smooth. The crucial point here is to show that, because
$F_1$ and $F_2$ have the same invariants, $h$ is in fact smooth.

\begin{claim}
  \label{claim2}
  The map $h$
  extends to an $S^1$\--equivariant diffeomorphism of a neighborhood of $F_2(M_2)$
  into a neighborhood of $F_1(M_1)$.
\end{claim}
The map $h$ is a already a homeomorphism. We need to show that it is a
local diffeomorphism everywhere. Let us denote by $c_i^j$ for $j \in \{1,\,2\}$
and $i \in \{1,\,\dots,\,m_f\}$ the focus-focus critical values of $F_j$. Again we
let $(e_1,\,e_2)$ be the canonical basis of $\R^2$ and let
$\ell_i^j\subset\R^2$ be the vertical half\--line starting at $c_i^j$
and pointing in the direction of $\epsilon_je_2$.

Since the tuple of heights of the focus\--focus points given by invariant (iv) are the same for both
systems, $g_1^{-1}(c_j^1)=g_2^{-1}(c_j^2)$ and
$g_1^{-1}(\ell_i^1)=g_2^{-1}(\ell_i^2)$. Hence $h$ is smooth away from
the union of all $\ell_i^2$, $i \in \{1,\,\dots,\,m_f\}$.

Let us now fix some $i\in\{1,\dots, m_f\}$ and let $\widetilde{U}_z$
be a small ball around a point $z\in\ell_i^2\setminus\{c_i^2\}$.  For
simplifying notations, we shall drop the various subscripts $i$.
Recall that $\widetilde{U}_z$ inherits from $F_2(M_2)$ an
integral-affine structure. Let $\phy_z:U_z\to\widetilde{U}_z$ be an
oriented affine chart, with $U_z$ a neighborhood of the origin in
$\R^2$, sending the vertical axis to $\ell^2$.
In order to show that $h$ is smooth on $\widetilde{U}_z$, we consider
the two halves of the ball $\widetilde{U}_z$~:
$$\widetilde{U}_z^+=\widetilde{U}_z\cap\{x\geq x_c\}\,\,\,\textup{ and}\,\,\,
\widetilde{U}_z^-=\widetilde{U}_z\cap\{x\leq x_c\}, 
$$
where $x_c$ is
the abscissa of $c^2$. Of course, the restrictions of $\phy_z$ to each
half $U_z\cap\{x\geq 0\}$ and $U_z\cap\{x\leq 0\}$ are admissible
affine charts for $\widetilde{U}_z^+$ and $\widetilde{U}_z^-$,
respectively.  Let us call these restrictions $\varphi^\pm_z$. Let
$y=h(z)$, $\widetilde{V}_y=h(\widetilde{U}_z)$ and
$\widetilde{V}_y^\pm=h(\widetilde{U}_z^\pm)$. Using the natural
integral\--affine structure on $F_1(M_1)$, we can similarly introduce an
affine chart $\psi_y$ for $\widetilde{V}_y$ and the corresponding
restrictions $\psi_y^\pm$.
We are now going to use the following facts~:
\begin{enumerate}
\item On each half $\widetilde{U}_z^+$ and $\widetilde{U}_z^-$, $h$ is
  an integral\--affine isomorphism~:
  $\widetilde{U}_z^\pm\to\widetilde{V}_y^\pm$.
\item The differential $\op{d}h$ is continuous on $\widetilde{U}$.
\end{enumerate}
The first fact implies, by definition, that the map
\[
\nu^\pm:=(\psi^\pm_y)^{-1} \circ h \circ \varphi^\pm_z,
\]
wherever defined, is of the form $ A^\pm(\cdot)+b^\pm, $ for some
matrix $A^\pm \in \textup{GL}(2,\, \mathbb{Z})$ and some constant
$b^\pm \in \mathbb{R}^2$. Evaluating the differentials at the origin,
we immediately deduce from the second fact that $A^+=A^-$. So,
$\nu^\pm$ should be just a translation. But $h$ itself being
continuous on the line segment $\ell^2$, we must have $b^+=b^-$. It
follows that, on $\widetilde{U}_z$, $h$ is equal to $\psi_y\circ
L\circ\phy_z^{-1}$, where $L$ is the affine transformation
$A^+(\cdot)+b^+=A^-(\cdot)+b^-$. So $h$ is indeed smooth on
$\widetilde{U}_z$.

We have left to show that $h$ is smooth at a focus\--focus critical
value $c^2$. The fact that we are assuming that invariant (ii) is the same
for both systems means that the corresponding 
symplectic invariants power series $(S)^{\infty}$ are the same
for both systems implies, by the semi\--global result of V\~u Ng\d oc \cite[Th.~2.1]{vungoc0},
that there exist a neighborhood $V(c^1)$ of $c^1$, a
neighborhood $V(c^2)$ of $c^2$, a semi\--global symplectomorphism $\phy:
F_1^{-1}(V(c^1)))\to F_2^{-1}(V(c^2))$ and a local diffeomorphism
$g:V(c^2)\to V(c^1)$ such that $F_1 = g\circ F_2\circ \phy$.
Now, from Lemma~\ref{rema:priviledged} we know that there exists a
privileged toric momentum map, c.f. Definition \ref{def:priviledged}, for each system above
the domain $V(c^j)\setminus\ell^j$, $j \in \{1,\,2\}$, c.f. Definition \ref{def:priviledged}. 
We denote by $\nu_1$ this momentum
map for the system induced by $F_1$, and $\nu_2$ for the system
induced by $F_2$.  Since $\nu_1$ and $\nu_2$ are semi-global
symplectic invariants, one has $\nu_1=\nu_2\circ\phy$.

By equation (\ref{for2:eq}) 
the focus\--focus critical values $c^1$ of
the semitoric systems
$F_1$ and and $c^2$ of $F_2$ have the same twisting\--index
invariant $k$ with
respect to the common polygon $\Delta$. In view of the
characterization of expresion (\ref{equ:twist}), we get
\begin{equation}
  \mu_1 = T^{k} \nu_1\qquad \text{ and } \qquad \mu_2=T^{k} \nu_2,
  \label{equ:mu}
\end{equation}
and therefore
\[
\mu_1=\mu_2\circ\phy.
\]
Thus we can write $g_1^{-1}F_1 = g_2^{-1}F_2\circ \phy$, or
\[
F_1=h\circ g^{-1} \circ F_1.
\]
Using that $F_1$ is a submersion on a neighborhood of any regular
torus, and the fact that $h \circ g^{-1}$ is smooth at the
corresponding regular values of $F_1$, we get that
\begin{equation}
  h \circ g^{-1} = \textup{Id} \qquad \text{ on } V(c^1).
  \label{equ:h}
\end{equation}
By continuity, this also holds at $c^1$. Hence $h= g$ is smooth at
$c^2$, which proves the claim.
\\~\\
{\em Step 2}: \emph{Local symplectomorphisms.} From step 1 we can
assume that
\begin{equation}
  F_1= g_1\circ \mu_1 \quad F_2 = g_1\circ\mu_2
  \label{equ:F}
\end{equation}
and hence $F_1(M_1)=F_2(M_2)$. In this second step we prove that this
common image can be covered by open sets $\Omega_\alpha$, above each
of which $F_1$ and $F_2$ are symplectically interwined.
  
  \begin{claim}
  \label{claim3}
  There exists a locally finite open cover
  $(\Omega_{\alpha})_{\alpha\in A}$ of $F_1(M_1)=F_2(M_2)$ such that
  \begin{enumerate}
  \item all $\Omega_{\alpha}$, $\alpha\in A$ are simply connected, and
    all intersections are simply connected;
  \item for each $\alpha\in A$, $\Omega_{\alpha}$ contains at most one
    critical value of rank $0$ of $F_i$, for any $i \in \{1,\,2\}$;
  \item for each $\alpha\in A$, there exists a symplectomorphism
    $\phy_\alpha: F_1^{-1}(\Omega_{\alpha})\to
    {F}_2^{-1}(\Omega_{\alpha})$ such that
    \[
    F_1 = {F}_2\circ \phy_\alpha \quad \text{ on }
    F_1^{-1}(\Omega_{\alpha}).
    \]
  \end{enumerate}
\end{claim}

We prove this claim next.  Recall that the toric momentum maps $\mu_1$
and $\mu_2$ have by hypothesis the same image, which is the polygon
$\Delta$. We can define an open cover $\tilde{\Omega}_\alpha$ of
$\Delta$ with open balls, satisfying points $(1)$ and $(2)$. When
the ball $\tilde{\Omega}_\alpha$ contains critical value of rank 0, we
may assume that this critical value is located at its
center. Similarly, when a ball contains critical values of rank 1, we
may assume that the set of rank 1 critical values in this ball is a
diameter.  Then we just define
\[
\Omega_{\alpha} = g_1(\tilde{\Omega}_\alpha).
\]
Notice that in doing so we ensure that the number and type of critical
values of $F_i$ in $\Omega_{\alpha}$ are the same for $i=1$ and $i=2$.
For proving point $(3)$ we distinguish four cases~:
\begin{itemize}
\item[(a)] $\Omega_{i,\alpha}$ contains no critical point of $F_i$;
\item [(b)] $\Omega_{i,\alpha}$ contains critical points of rank 1, but
  not of rank 0;
\item [(c)] $\Omega_{i,\alpha}$ contains a rank 0 critical point, of
  elliptic-elliptic type;
\item [(d)] $\Omega_{i,\alpha}$ contains a rank 0 critical point, of
  focus-focus type.
\end{itemize}
The reasoning for all cases follows the same lines, but we keep these
cases separated for the sake of clarity.

\paragraph{Case (a).} Let us fix a point $c_\alpha\in\Omega_\alpha$. By
Liouville\--Mineur\--Arnold theorem \cite{duistermaat}, applied for each momentum map $F_i$
over the simply connected open set $\Omega_\alpha$, there exists a
symplectomorphism $\tilde{\phy}_{i,\alpha}:F_i^{-1}(\Omega_\alpha)\to
\op{T}^*\T^{2}$ and a local diffeomorphism $h_i:(\R^2,\,0)\to(\R^2,\,c_\alpha)$
such that
\[
F_i=h_i(\xi_1,\xi_2)\circ\tilde{\phy}_{i,\alpha}.
\]
Here we use the notation $(x_1,\,x_2,\,\xi_1,\,\xi_2)$ for canonical
coordinates in $\op{T}^*\T^2$, where the zero section is given by
$\{\xi_1=\xi_2=0\}$.

In fact because of equation~\eqref{equ:F},
$\mu_i=g_1^{-1}h_i(\xi_1,\,\xi_2)\circ\tilde{\phy}_{i,\alpha}$. Since
both $\mu_i$ and $(\xi_1,\,\xi_2)$ are toric momentum map, this implies
that $g_1^{-1}h_i$ is an affine map with a linear part
$B_i\in\op{GL}(2,\,\Z)$. Now we can define a linear
symplectomorphism in a block\--diagonal way 
$$S_i=
\begin{pmatrix}
  \trsp B_i&0\\0& B_i^{-1}.
\end{pmatrix}
$$
Obviously $(\xi_1,\xi_2)\circ S_i=B_i^{-1}\circ(\xi_1,\xi_2)$. From
now on we replace $\tilde{\phy}_{i,\alpha}$ by $S_i\circ
\tilde{\phy}_{i,\alpha}$, which reduces us to the case
$B_i=\textup{Id}$.

Now, let
$\phy_\alpha:=\tilde{\phy}_{2,\alpha}^{-1}\circ\tilde{\phy}_{1,\alpha}$. We
have the relation
\[
F_1 = (h_1 h_2^{-1})\circ
F_2\circ\tilde{\phy}_{2,\alpha}^{-1}\circ\tilde{\phy}_{1,\alpha} = g_1
(g_1^{-1}h_1)(g_1^{-1}h_2)^{-1} g_1^{-1}\circ\phy_\alpha.
\]
The affine diffeomorphism $(g_1^{-1}h_1)(g_1^{-1}h_2)^{-1}$ is tangent
to the identity and fixes the point $c_\alpha$; hence it is the
identity, and we obtain, as required~:
\[
F_1 = F_2\circ\phy_\alpha, \quad \text{ on } F_1^{-1}(\Omega_\alpha).
\]

\paragraph{Case (b).} Above $\Omega_\alpha$, the momentum map has
singularities, so we cannot apply the action-angle theorem. However,
there is still a $\T^2$-action, and it is well known that an
``action-angle with elliptic singularities'' theorem holds
(see~\cite{dufour-molino} or ~\cite{miranda-zung}).  Precisely, we fix
a point $c_\alpha\in\Omega_\alpha$ that is a critical value of $F_1$
and $F_2$, and then for each $i \in \{1,\,2\}$, there exists a symplectomorphism
$\tilde{\phy}_{i,\alpha}: F_i^{-1}(\Omega_{\alpha})\to \op{T}^*\R\times
\op{T}^*\T^1$ and a local, orientation preserving diffeomorphism
$h_i:(\R^2,\,0)\to(\R^2,\,c_\alpha)$ such that
\[
F_i=h_i(q_1,\xi_2)\circ\tilde{\phy}_{i,\alpha}.
\]
Here $\op{T}^*\R$ has canonical coordinates $(x_1,\,\xi_1)$,
$q_1=(x_1^2+\xi_1^2)/2$, and $\op{T}^*\T^1$ has canonical coordinates
$(x_2,\, \xi_2)$. As before, one has
$$
\mu_i=g_1^{-1}h_i(q_1,\,\xi_2)\circ\tilde{\phy}_{i,\alpha}, 
$$
and
$g_1^{-1}h_i$ is an affine map with a linear part in
$\textup{SL}(2,\Z)$.

Since $h_i$ must preserve the set of critical values, it must send the
vertical axis ``$q_1=0$''$\subset\R^2$ to the set of critical values
in $\Omega_\alpha$. Hence $g_1^{-1}h_i$ sends the vertical axis to the
corresponding diameter in $\tilde{\Omega}_\alpha$. Moreover, since by
hypothesis the images of $\mu_1$ and $\mu_2$ are the same, the set
``$q_1>0$'' corresponds via $g_1^{-1}h_1$ and $g_1^{-1}h_2$ to the
same half of the ball $\tilde{\Omega}_\alpha$. In other words, the
vector $e_2=(0,1)$ is an eigenvector for the linear part $B$ of
$(g_1^{-1}h_1)^{-1}g_1^{-1}h_2$, with eigenvalue 1.

Therefore, $B$ is of the form $
\begin{pmatrix}
  1 & 0\\k & 1
\end{pmatrix}
$, for some integer $k\in\Z$. Now, consider the map
$S(x_1,\, x_2,\,\xi_1,\,\xi_2)=(x_1',\,x_2',\,\xi_1',\,\xi_2')$ given by
\begin{equation}
  \label{equ:S}
  \left\{
    \begin{aligned}
      (x_1'+i\xi_1')&=\textup{e}^{ikx_2}(x_1+i\xi_1)\\
      x_2'&=x_2\\
      \xi_2'&= \xi_2+k q_1.
    \end{aligned}\right.
\end{equation}
In complex coordinates, 
$$\DD{}\xi_1\wedge \DD{}
x_1=\frac{1}{2i}\DD{} z_1\wedge \DD{}\bar{z}_1,$$
so it is easy to
check that $S$ is symplectic. Moreover,
\[
(q_1,\xi_2)\circ S =
\begin{pmatrix}
  1 & 0\\k & 1
\end{pmatrix} (q_1,\xi_2) = B(q_1,\xi_2).
\]
We can write
\[
F_1 = h_1 B^{-1}(q_1,\xi_2)\circ S \circ \tilde{\phy}_{1,\alpha},
\]
and hence, letting $\phy_\alpha:=\tilde{\phy}_{2,\alpha}^{-1}\circ
S\circ \tilde{\phy}_{\alpha}$,
\[
F_1 = (h_1 B^{-1} h_2^{-1})\circ F_2\circ\phy_\alpha.
\]

Consider the affine map $(g_1^{-1}h_1)^{-1}g_1^{-1}h_2 B^{-1}$. Its
linear part is the identity, and it fixed the origin; thus it is the
identity. This implies that $h_1 B^{-1} h_2^{-1}=\textup{Id}$.

\paragraph{Case (c).} Using Eliasson's local normal form for
elliptic-elliptic singularities, we have the existence of a
symplectomorphism $\tilde{\phy}_\alpha$ and a local diffeomorphism $h$
such that
\[
F_1=h\circ F_2\circ\tilde{\phy}_\alpha.
\]
Again, because of equation~\eqref{equ:F}, $\mu_2=g_1^{-1}hg_1\circ
\mu_1\circ\tilde{\phy}_\alpha$, and
$g_1^{-1}hg_1\in\textup{GL}(2,\Z)$. But since the image of $\mu_1$ and
the image of $\mu_2$ are the same, then $g_1^{-1}hg_1$ must send the
corner of the polygon to itself. Since it is a Delzant corner,
$g_1^{-1}hg_1$ must be the identity.

\paragraph{Case (d).} From~(\ref{equ:F}), the momentum maps $F_1$ and
$F_2$ have the same focus\--focus critical values $c_1,\dots,c_{m_f}$.  We
wish here to interwine both systems above a small neighborhood of each
$c_i$. In order to ease notations, let us drop the subscript $i$, as
the construction can be repeated for each focus-focus point.

The behaviour of the system in a neighborhood of $c$ is given by V\~u
Ng\d oc's theorem, which we already used for the proof of
Claim~\ref{claim2}.  Precisely \cite[Th.~2.1]{vungoc} gives the
existence of a neighborhood $V(c)$ of $c$ together with an equivariant
symplectomorphism
  \begin{eqnarray} \label{symplecto1} \psi \colon F_1^{-1}(V(c)) \to
    F_2^{-1}(V(c))
  \end{eqnarray}
  and a diffeomorphism $g$ such that
  \begin{eqnarray} \label{main} F_1=g \circ F_2 \circ \psi.
  \end{eqnarray}
Now, we may argue exactly as
in~(\ref{equ:mu})~--~(\ref{equ:h}), keeping in mind that we are now in
the case where $h=\textup{Id}$. Hence $g$ also must be the identity map.

This concludes the proof of Claim~\ref{claim3} and hence Step 2.
\\~\\
{\em Step 3:} \emph{Local to Global}.  In this last step is to glue
together the local symplectomorphisms of Step 2, thus constructing a
global symplectomorphism $\phi:M_1\to M_2$ such that
$F_1=F_2\circ\phi$.
For technical reasons, we introduce a slightly smaller open cover
that $(\Omega_\alpha)_{\alpha\in A}$.

\begin{claim}
\label{claim4}
There exists an open cover $(\Omega'_\alpha)_{\alpha\in A}$ of $F_1(M_1)=F_2(M_2)$ such that 
\begin{itemize}
\item[(i)]
$\Omega'_\alpha\Subset\Omega_\alpha$,
\item[(ii)]
$(\Omega'_\alpha)_{\alpha\in A}$
verifies the properties
of Claim~\ref{claim3}, i.e. if we replace $\Omega_{\alpha}$ by $\Omega'_{\alpha}$
therein, Claim~\ref{claim3} holds.
\item[(iii)]
If $\alpha,\, \beta \in A$ are such that $\Omega'_{\alpha}\cap\Omega'_{\beta}\neq\emptyset$, there
  exists a smooth symplectomorphism: 
  $$\phy_{(\alpha,\beta)}:
  F_1^{-1}(\Omega'_{\alpha}\cup\Omega'_{\beta})\to
  F_2^{-1}(\Omega'_{\alpha}\cup\Omega'_{\beta})
  $$
   such that
  \begin{enumerate}
  \item $(\phy_{(\alpha,\beta)})|_{F^{-1}(\Omega'_\alpha)}=\phy_\alpha$;
  \item $F_1 = \phy_{(\alpha,\beta)}^* F_2 \quad \text{ on }
    F_1^{-1}(\Omega'_{\alpha}\cup\Omega'_{\beta}).$
\end{enumerate}
\end{itemize}
\end{claim}

The proof of this claim uses the Hamiltonian structure of the group
of symplectomorphisms preserving homogeneous momentum maps, which we
state below. It is due to Miranda-Zung~\cite{miranda-zung}.

First we introduce some notation. Let $h_1,\dots h_n$ be $n$
Poisson-commuting functions$: \R^{2n}\to\R$.  Suppose that $\psi
\colon (\R^{2n},\,0) \to (\R^{2n},\,0)$ is a local symplectomorphism
of $\R^{2n}$ which preserves the smooth map $h=(h_1,\ldots,h_n)$.

Let $\op{Sympl}(\R^{2n})$ be the group of symplectomorphisms of $\R^{2n}$.
Consider the set
$$
\Gamma:=\{\phi \in \op{Sympl}(\R^{2n}) \, | \, \phi(0)=0,\,\,\, h
\circ \phi =h\},
$$
and let $\Gamma_0$ stand for the path\--connected component of the
identity of $\Gamma$. Let $\got{g}$ be the Lie algebra of germs
of Hamiltonian vector fields tangent to the fibration $\mathcal{F}$
given by $h$.

Let $\op{exp} \colon \got{g} \to \Gamma_0$ be the exponential
mapping determined by the time\--1 flow of a vector field $X \in
\got{g}$. More precisely, the time\--s flow $\phi^s_X$ of $X$
preserves $h$ because $X$ is tangent to $\mathcal{F}$, and it
preserves the symplectic form because $X$ is a Hamiltonian vector
field.  The mapping $\phi^s_{X}$ fixes the origin because $X$ vanishes
there.  Hence $\phi^s_{X_G}$ is contained in $\Gamma_0 \subset
\Gamma$ since $\phi^0_X$ is the identity map.
\begin{claim}
  \label{claim:flow}
  Suppose that each $h_i$ is a homogeneous function, meaning that for
  each $i \in \{1,\ldots,n\}$ there exists an integer $k_i \ge 0$ such
  that $h_i(t\,x)=t^{k_i}\,h_i(x)$ for all $x \in \R^n$. Then:
  \begin{enumerate}
  \item The linear part $\psi^{(1)}$ of $\psi$ is a symplectomorphism
    which preserves $h$.
  \item There is a vector field contained in $\mathfrak{g}$ such that
    its time\--1 map is $\psi^{(1)} \circ \psi^{-1}$. Moreover, for
    each vector field $X$ fulfilling this condition there is a unique
    local smooth function $\Psi: (\R^{2n},\,0) \rightarrow \R$ vanishing
    at $0$
    such that $X=X_{\Psi}$.
  \end{enumerate}
\end{claim}

Although not explicitely written in~\cite{miranda-zung}, the proof of
this claim is a minor extension of the case treated therein,
where all $h_i$ are quadratic functions.  For completeness, we have
included a proof as an appendix.
~\\~\\
We turn now to the proof of Claim~\ref{claim4}.  Because of
Claim~\ref{claim3}, there cannot be a critical value of $F_1$ of rank
zero in the intersection $\Omega_{\alpha}\cap\Omega_{\beta}$. Hence we
have two cases to consider~:
\begin{itemize}
\item[(1)] the intersection contains no critical value;
\item[(2)] the intersection contains critical values of rank one.
\end{itemize}

  \paragraph{Case (1).}
  Let $\varphi_{\alpha\beta}=\varphi_\alpha\varphi^{-1}_\beta$. It is
  well defined as a symplectomorphism of
  $M_{\alpha\beta}:=F_2^{-1}(\Omega_\alpha\cap\Omega_\beta)$ into
  itself. Moreover, $F_2^*\varphi_{\alpha\beta}=F_2$. Since $F_2$ is
  regular on $M_{\alpha\beta}$ (and $\Omega_\alpha\cap\Omega_\beta$ is
  simply connected), one can invoke the Liouville\--Mineur\--Arnold
  theorem~\cite{duistermaat} and assume that
  $M_{\alpha\beta}=\T^n\times D$, with corresponding angle-action
  coordinates $(x,\, \xi)$, where $D$ is some simply connected open
  subset of $\R^n$, in such a way that $F_2$ depends only the the
  $\xi$ variables. 
  
  The symplectomorphism $\phy_{\alpha\beta}$
  preserves the linear momentum map $\xi=(\xi_1,\, \xi_2)$, so we may
  apply Claim~\ref{claim:flow} and obtain
   a smooth function $h_{\alpha\beta}$ on
  $\Omega_\alpha\cap\Omega_\beta$ such that $\varphi_{\alpha\beta}$ is
  the time\--1 Hamiltonian flow of $h_{\alpha\beta}\circ F_2$. Let
  $\chi$ be a smooth function on $\Omega_\alpha\cup\Omega_\beta$
  vanishing outside $\Omega_\alpha\cap\Omega_\beta$ and identically
  equal to 1 in $\Omega'_\alpha\cap\Omega'_\beta$.  Thus we may
  construct a smooth function $\tilde{h}_{\alpha\beta}=\chi
  h_{\alpha\beta}$ on $\Omega_\alpha\cup\Omega_\beta$ whose
  restriction to a slightly smaller open set $\Omega'_\alpha\cap\Omega'_\beta$ is precisely
  $h_{\alpha\beta}$, where $\Omega'_{\alpha} \Subset \Omega_{\alpha}$
  and $\Omega'_{\beta} \Subset \Omega_{\beta}$ are chosen precisely
  so that this condition is satisfied. 
  Let $\tilde{\varphi}_{\alpha\beta}$ be the time-1
  Hamiltonian flow of $\tilde{h}_{\alpha\beta}\circ F_2$. It is
  defined on $F_2^{-1}(\Omega_\alpha\cup\Omega_\beta)$, and equal to
  $\phy_{\alpha\beta}$ on $F_2^{-1}(\Omega'_\alpha\cap\Omega'_\beta)$.

  Now consider the map $\psi$ defined on
  $F_1^{-1}(\Omega'_\alpha\cup\Omega'_\beta)$ by
$$
\psi(m) =
\begin{cases}
  \varphi_{\alpha} (m) & \textrm{if } m\in F_1^{-1}(\Omega'_\alpha)\\
  \tilde{\varphi}_{\alpha\beta}\circ\varphi_{\beta}(m) & \textrm{if }
  m \in F_1^{-1}(\Omega'_\beta).
\end{cases}
$$
It is well\--defined because on
$F_1^{-1}(\Omega'_\alpha\cap\Omega'_\beta)$ one has 
$$
\tilde{\varphi}_{\alpha\beta}\circ\varphi_{\beta}(m) =
{\varphi}_{\alpha\beta}\circ\varphi_{\beta}(m) = \varphi_\alpha(m).
$$
Then $\psi$ is a smooth symplectomorphism:
$F_1^{-1}(\Omega'_\alpha\cup\Omega'_\beta)\to
F_2^{-1}(\Omega'_\alpha\cup\Omega'_\beta)$.  Moreover, since
$F_2=\tilde{\varphi}_{\alpha\beta}^*F_2$, one has $ F_1 = \psi^* F_2$.
Thus, in this case, we may let $\phy_{(\alpha,\beta)}=\psi$.

\paragraph{Case (2).} Again we let
$\varphi_{\alpha\beta}=\varphi_\alpha\varphi^{-1}_\beta$, a
symplectomorphism of
$M_{\alpha\beta}:=F_2^{-1}(\Omega_{\alpha}\cap\Omega_{\beta})$ into
itself, such that $F_2\circ\varphi_{\alpha\beta}=F_2$.

By Miranda\--Zung's result \cite[Th.~ 2.1]{miranda-zung} 
(or Dufour\--Molino \cite{dufour-molino} or
Eliasson \cite{eliasson-these}), the foliation above $\Omega_{\alpha}\cap\Omega_{\beta}$
is symplectomorphic to the linear model $(q_1,\xi_2)$ on $\op{T}^*\R\times
\op{T}^*\T^1$, with 
$$q_1(x_1,\,\xi_1,\,x_2,\,\xi_2)=x_1^2+\xi_1^2.$$ 
This means
that there exists a symplectomorphism $\chi:M_{\alpha\beta}\to
\op{T}^*\R\times \op{T}^*\T^1$ and a diffeomorphism $h$ of
$\chi(\Omega_{\alpha}\cap\Omega_{\beta})$ such that
\[
F_2\circ\chi^{-1} = h(q_1,\,\xi_2).
\]

Hence we find that
\[
h(q_1,\,\xi_2)\circ\chi = h(q_1,\,\xi_2)\circ\chi\circ\phy_{\alpha\beta}.
\]
By Claim~\ref{claim:flow}, there exists a smooth function
$\hat{h}_{\alpha\beta}=\hat{h}_{\alpha\beta}(q_1,\xi_2)$ whose
Hamiltonian flow connects
$\hat{\phy_{\alpha\beta}}:=\chi\circ\phy_{\alpha\beta}\circ\chi^{-1}$
to its linear part of at the origin. Now, it is easy to see that any
linear symplectomorphism preserving the moment map $(q_1,\,\xi_2)$ is
the time-1 flow of some linear function
$q_{\alpha\beta}(q_1,\xi_2)$. Since any two functions of $(q_1,\,\xi_2)$
commute, the time\--1 Hamiltonian flow of the half sum
$(\hat{h}_{\alpha\beta} + q_{\alpha\beta})/2\circ(q_1,\, \xi_2)$ is
precisely $\hat{\phy_{\alpha\beta}}$. By naturality, the time-1
Hamiltonian flow of $$h_{\alpha\beta}\circ F_2=(\hat{h}_{\alpha\beta} +
q_{\alpha\beta})/2 \circ (q_1,\,\xi_2)\circ \chi,$$ defined on
$\Omega_\alpha\cap\Omega_\beta$, is precisely $\phy_{\alpha\beta}$.
We now conclude as in case a).
\\~\\
{\em Conclusion:}
It follows from Claim~\ref{claim3} and the second point of
Claim~\ref{claim4} that for any finite subset $A'\subset A$, there
exists a symplectomorphism $\phi_{A'}: F_1^{-1}(\Omega_{A'})\to
F_2^{-1}(\Omega_{A'})$, where
\[
\Omega_{A'}:=\bigcup_{\alpha\in A'}\Omega'_{\alpha},
\]
such that
\[
F_1 = F_2 \circ \phi_{A'}.
\]
Moreover, from the first point of Claim~\ref{claim4} we see that of
$A"\subset A$ is another finite subset containing $A'$, then one can
choose $\phi_{A"}$ such that $(\phi_{A"})|_{\Omega_{A'}}=\phi_{A'}$.

Let $(A_n)_{n\in\mathbb{N}}$ be an increasing sequence of finite
subsets of $A$ whose union is $A$. The projective limit of the
corresponding sequence $(\phi_{A_n})$ is a symplectomorphism
$\phi:M_1\to M_2$ such that $F_1=F_2\circ \phi$, which finally proves
the theorem.

\section{Proof of Miranda\--Zung's lemma for homogeneous maps}
\label{sec:miranda}

Let $\phi \in \Gamma$. Let $g_t \in \op{C}^{\infty}(\R^{2n})$
  the expansion mapping
  $g_t(x_1,\ldots,x_{2n})=t\,(x_1,\ldots,x_{2n})$ for each $t \in
  \R$. Consider the deformation given by the family
  $\{S^{\psi}_t(x)\}_{t \in [0,\,2)}$ defined by
  \begin{eqnarray}
    S_t^{\psi}(x)=\begin{cases}{\displaystyle{1/t\,\,(\psi\circ g_t})(x)} &\quad
      t\in(0,2]\\ \hfill \psi^{(1)}(x) & \quad t=0.
    \end{cases}
  \end{eqnarray}

  This deformation is usually called ``Alexander's trick" and it is
  well\--known to be smooth \cite{hirsch}. We have to check that the
  deformation takes place inside of $\Gamma$, which amounts to
  checking that $h \circ S^{\psi}_t=h$ for all $t \in [0,\,2]$, and
  that it is symplectic, i.e.  $(S^{\psi}_t)^*\omega=\omega$ for all
  $t \in [0,\,2]$.

  In order to check this, let us assume that $t \neq 0$ in what
  follows. Indeed, we have that\footnote{this elementary computation is the only difference with the
  proof of Corollary 3.4 in the article \cite{miranda-zung} of
  Miranda\--Zung, where the index $k_i$ equals $2$ therein because
  $h_i$ is a homogeneous quadratic polynomial of degree $2$. After we became
  aware of this, we asked E. Miranda for confirmation, and we thanks her for it.}

  \begin{eqnarray}
    h_i \circ \frac{1}{t}(\psi \circ g_t)(x)
    =\Big( \frac{1}{t} \Big)^{k_i}h_i\, (\psi \circ g_t)(x)   
    = \frac{1}{t^{k_i}}\,h_i(g_t(x)) = \frac{1}{t^{k_i}}\,h_i(t\,x)  = \frac{1}{t^{k_i}}\,t^{k_i}\,h_i(x), \label{eqa}
  \end{eqnarray}
  where in the first equality we have used that $h_i$ is homogeneous of degree $k_i$,
  in the second that $\psi \in \Gamma$ and hence $h_i \circ \psi=h$, and in the fourth 
  again that  $h_i$ is homogeneous of degree $k_i$.

  On the other hand, $g_t^*\omega=t^2\, \omega$ since $\omega$ is a
  $2$\--form, and therefore
  \begin{eqnarray} \label{eqb} (S^{\psi}_t)^*\omega=1/t \, (\psi \circ
    g_t)^*\omega=\omega.
  \end{eqnarray}
  It follows from (\ref{eqa}) and (\ref{eqb}) that $S^{\psi}_t \in
  \Gamma$ for all $t \in (0,\,2]$. Because the definition given
  by $\{S_t^{\psi}\}_{t \in (0,\,2]}$ is smooth, $S_0^{\psi} \in
  \Gamma$, and in particular $\psi^{(1)} \in \Gamma$, which
  proves (1).

  Let $\Gamma_0$ be be the path\--connected component of the
  identity of $\Gamma$.  To conclude the proof it suffices to show
  that $\psi^{(1)} \circ \psi^{-1} \in \Gamma_0$, because once we
  know this (2) will follow from Theorem 3.2 in Miranda\--Zung
\footnote{Miranda\--Zung's theorem is stated for $h$ with quadratic
components; however their proof works for any smooth $h$ so long as
$\got{g}$ is abelian; they prove that $\got{g}$ is abelian in Sublemma
3.1 for the case of $h$ quadratic in \cite{miranda-zung}; their proof
immediately applies to homogenous $h$, and it is actually true in much
greater generality.}
  \cite{miranda-zung}:
  {\em The exponential $\exp\colon\mathfrak{g}\longrightarrow \mathcal
    G_0$ is a surjective group homomorphism, and moreover there is an
    explicit right inverse given by
$\phi\in\mathcal G_0\longmapsto \int_0^1 X_t \, \DD{}t \in \mathfrak{g}$
\noindent where $X_t\in \mathfrak{g}$ is defined by
$X_t(R_t)=\frac{\DD{} R_t}{\DD{}t}$
for any $\op{C}^1$ path $R_t$ contained in $\Gamma_0$ connecting the
identity to $\phi$.}

As in Miranda\--Zung's proof for the case that $h$ is quadratic
homogeneous, we take $R_t=\psi^{(1)} \circ S_t^{(\psi^{-1})}$, $t \in
[0,\,1]$.  The path $\{R_t\}_{t \in [0,\,1]} \subset \Gamma_0$ is
connects the identity to $\psi^{(1)} \circ S_t^{\psi^{-1}}$.  Hence by
the result above there exists a vector field $X$ whose time\--1 map is
$\psi^{(1)} \circ \psi^{-1}$ and a unique Hamiltonian mapping $\Psi$
which vanishes at the origin such that $X=X_{\Psi}$, and (2) follows.
\\
\\
{\em Acknowledgements}. A. Pelayo thanks the Department of Mathematics 
at MIT and Chair Sipser for the excellent resources he was
provided with while he was a postdoc there.

\noindent
\\
\\
Alvaro Pelayo\\
Massachusetts Institute of Technology\\
MIT room 2\--282\\
77 Massachusetts Avenue\\
Cambridge MA 02139-4307 (USA)
\\
{\em and}
\\
University of California\---Berkeley \\
Mathematics Department \\
970 Evans Hall 3840 \\
Berkeley, CA 94720-3840, USA.\\
{\em E\--mail}: {apelayo@math.mit.edu}

\bigskip\noindent

\noindent
San V\~u Ng\d oc\\
Universit\'e de Rennes 1\\
Math\'ematiques Unit\'e\\
Campus de Beaulieu\\
35042 Rennes cedex (France)\\
{\em E-mail:} {san.vu-ngoc@univ-rennes1.fr}

\end{document}